\documentclass[11pt]{article}
\usepackage{amsmath,amssymb,amsthm,graphicx,fixmath}

\usepackage{algorithmic}
\usepackage{algorithm}
\usepackage{setspace}
\usepackage{ifthen}
\usepackage{subfigure}
\usepackage{tikz}

\newcommand*{\Positives}{\mathbb{N}}

\providecommand{\remove}[1]{}

\theoremstyle{plain}
\newtheorem{theorem}{Theorem}[section]
\newtheorem{lemma}[theorem]{Lemma}

\newtheorem{proposition}[theorem]{Proposition}
\newtheorem{claim}[theorem]{Claim}
\newtheorem{corollary}[theorem]{Corollary}

\newtheorem{observation}[theorem]{Observation}
\theoremstyle{definition}
\newtheorem{definition}[theorem]{Definition}
\theoremstyle{remark}
\newtheorem{remark}[theorem]{Remark}

\newcommand{\E}{\mathcal{E}}

\def\famB{\mathcal{B}}
\def\famA{\mathcal{A}}
\def\famC{\mathcal{C}}
\def\famD{\mathcal{D}}
\def\famF{\mathcal{F}}
\def\famG{\mathcal{G}}
\def\famI{\mathcal{I}}
\def\famJ{\mathcal{J}}
\def\famH{\mathcal{H}}
\def\famE{\mathcal{E}}
\def\famY{\mathcal{Y}}
\def\famM{\mathcal{M}}
\def\famK{\mathcal{K}}
\def\famL{\mathcal{L}}

\def\famS{\mathcal{S}}
\def\famV{\mathcal{V}}
\def\famW{\mathcal{W}}

\def\famX{\mathcal{X}}

\newcommand{\li}{\ell}

\newcommand{\cf}{{{\chi_{\text{\textup{cf}}}}}}

\newcommand{\pn}{{{\chi_{\text{\textup{cf}}}}}}

\newcommand{\kpn}{{{\chi_{\text{\textup{k-cf}}}}}}
\newcommand{\norm}[1]{\left\lVert#1\right\rVert}

\newcommand{\cardin}[1]{\lvert {#1} \rvert}
\newcommand{\card}[1]{\lvert {#1} \rvert}

\newcommand{\mn}[0]{\medskip\noindent}

\begin{document}

\title{Conflict-Free Coloring of String Graphs}

\author{%
Chaya Keller\thanks{Department of Mathematics, Technion --- Israel Institute of Technology, Haifa 32000, Israel. 
Research partially supported by Grant 635/16 from the Israel Science Foundation, by the Shulamit Aloni Post-Doctoral Fellowship of the Israeli Ministry of Science and Technology, and by the Kreitman Foundation Post-Doctoral Fellowship.
\texttt{chayak@technion.ac.il}
}
\and Alexandre Rok\thanks{Department of Mathematics, Ben-Gurion University of the Negev,
Be'er Sheva 84105, Israel. Research partially supported by Grant 635/16 from the Israel Science Foundation.
\texttt{roky3090@gmail.com}
}
\and Shakhar Smorodinsky\thanks{Department of Mathematics, Ben-Gurion University of the Negev,
Be'er Sheva 84105, Israel. Research partially supported by Grant 635/16 from the Israel Science Foundation.
\texttt{shakhar@math.bgu.ac.il}
}%
}

\date{}
\maketitle

\begin{abstract}
Conflict-free coloring (in short, CF-coloring) of a graph $G = (V,E)$
is a coloring of $V$ such that the  neighborhood of each vertex contains a vertex whose color differs
from the color of any other vertex in that neighborhood. Bounds on CF-chromatic numbers have been studied both for general graphs and for intersection graphs of geometric shapes. In this paper we obtain such bounds for several classes of string graphs, i.e., intersection graphs of curves in the plane:


(i) We provide a general upper bound of $O(\chi(G)^2 \log n)$ on the CF-chromatic number of any string graph $G$ with $n$ vertices in terms of the classical chromatic number $\chi(G)$. This result stands in contrast to general graphs where the CF-chromatic number can be $\Omega(\sqrt{n})$ already for bipartite graphs.

(ii) For some central classes of string graphs, the
CF-chromatic number is as large as $\Theta(\sqrt{n})$, which was shown to be the upper bound for any graph even in the non-geometric context. For several such classes (e.g., intersection graphs of frames) we prove a tight bound of $\Theta(\log n)$ with respect to the relaxed notion of $k$-CF-coloring (in which the punctured neighborhood of each vertex contains a color that appears at most $k$ times), for a small constant $k$.

(iii) We obtain a general upper bound on the $k$-CF-chromatic number of arbitrary hypergraphs: Any hypergraph with $m$ hyperedges can be $k$-CF colored with $\tilde{O}(m^{\frac{1}{k+1}})$ colors. This bound, which extends a bound of Pach and Tardos (2009), is tight for some string graphs, up to a logarithmic factor.

(iv) Our fourth result concerns circle graphs in which coloring problems are motivated by VLSI designs. We prove a tight bound of $\Theta(\log n)$ on the CF-chromatic number of circle graphs, and an upper bound of $O(\log^{3} n)$ for a wider class of string graphs that contains circle graphs, namely, intersection graphs of grounded L-shapes.

\end{abstract}


\section{Introduction}

\subsection{Background}

\paragraph{Conflict-free coloring and $k$-conflict-free coloring of hypergraphs.}
\begin{definition}
Let $H=(V,\E)$ be a hypergraph and let $c$ be a coloring $c \colon V \rightarrow \Positives$.
We say that $c$ is a \emph{conflict-free coloring} (CF-coloring in short) if every hyperedge $S \in \E$ contains a uniquely colored vertex, i.e., if there exists a color $i \in \Positives$ such that $\card{S \cap c^{-1}(i)}=1$.

For $k \in \mathbb{N}$, we say that $c$ is a $k$-CF-coloring if every hyperedge $S \in \E$ contains a color that appears at most $k$ times, i.e., if there exists a color $i \in \Positives$ such that $1 \leq \card{S \cap c^{-1}(i)} \leq k$.

The minimum integer $\ell$ such that $H$ admits a CF-coloring (respectively, $k$-CF-coloring) with $\ell$ colors is called the CF-chromatic (respectively, $k$-CF-chromatic) number of $H$ and is denoted by $\chi_{\mathrm{cf}}(H)$ (respectively, $\chi_{\text{\textup{k-cf}}}(H)$).
\end{definition}

CF-coloring of hypergraphs was introduced at FOCS'2002 by Even et al.~\cite{ELRS} and in the Ph.D. thesis of Smorodinsky \cite{SmPHD}.
The original motivation to study CF-coloring derived from spectrum allocation in wireless networks.
Since then, the notion of CF-coloring has attracted significant attention from researchers, and its different aspects and variants have been studied in dozens of papers. See, e.g.,~\cite{AS08,Chan12,cf7,CKS2009talg,GR15,hks09,CFPT09} for a sample of various aspects of CF-coloring, and the survey~\cite{CF-survey} and the references therein for more on CF-colorings and their applied and theoretical aspects.


In 2003, Har-Peled and Smorodinsky~\cite{HS02,SmPHD} introduced the notion of $k$-CF coloring.
As in the first papers on CF-coloring, the
motivationto study $k$-CF coloring stems from the allocation of frequencies in cellular networks, since in reality, the interference between conflicting antennas
is a function of the number of such antennas.

\paragraph{Conflict-free coloring of graphs}

For a simple undirected graph $G=(V,E)$ and a vertex $v \in V$, the \emph{punctured neighborhood} of $v$ in $G$
is $N_G(v)=\{u \in V: (v,u) \in E\}$. 
A coloring $C$ of $V$ is called 
a \emph{punctured CF-coloring}, or simply \emph{CF-coloring},\footnote{We note that the
related concept of \emph{closed} CF-coloring, in which the closed neighborhood of each vertex is required to contain a uniquely
colored vertex, has also been studied extensively (see, e.g.,~\cite{AADFGHKS17,GST14,CFPT09}). In this paper we consider only punctured
CF-coloring, and hence throughout the paper we mostly shorten `punctured CF-coloring' into `CF-coloring'.}
if for any $v \in G$,
the set $N_G(v)$ contains a uniquely colored vertex. Similarly, $C$ is called a punctured
$k$-CF-coloring, or simply a $k$-CF-coloring, if for any $v \in G$, the set $N_G(v)$ contains a color that appears at most $k$ times.

The minimum number of colors in a punctured CF-coloring of $G$ is denoted by
$\pn(G)$. The notion $\kpn(G)$ is defined similarly.
Note that a punctured CF-coloring of a graph $G=(V,E)$ is
a CF-coloring of the hypergraph on the same vertex set whose hyperedges are all sets of the form $N_G(v)$, for all $v \in V$, which we call \emph{the neighborhood hypergraph of $G$}.

CF-coloring of graphs was introduced in 2009 by Cheilaris~\cite{CheilarisCUNYthesis2009} and Pach and Tardos~\cite{CFPT09}
and has been studied extensively since then. In particular, for general graphs on $n$ vertices, Cheilaris~\cite{CheilarisCUNYthesis2009} and
Pach and Tardos~\cite{CFPT09} obtained the upper bound $\pn(G) = O(\sqrt{n})$ for punctured CF-coloring. This bound is attained as a lower bound for various classes of graphs, as we discuss later.
Gargano and Rescigno~\cite{GR15} showed that punctured CF-coloring is hard to approximate within a factor of $n^{\frac {1}{2}-\epsilon}$.


\paragraph{Conflict-free coloring of geometrically defined graphs.}
Recently, several papers have considered CF-coloring of graphs with a geometric nature.  Abel et al. \cite{AADFGHKS17} studied a related notion of closed CF-coloring of planar graphs. Fekete and Keldenich \cite{FK17} studied intersection graphs of unit discs and of unit squares, that is, graphs whose vertices are the geometric objects, and two objects are adjacent if their intersection is non-empty.
Keller and Smorodinsky~\cite{KS18} studied intersection graphs $G$ of pseudo-discs (i.e., simple Jordan regions such that the boundaries of any
two such regions intersect in at most two points) and proved the asymptotically tight bound $\pn(G) = O(\log n)$.
Very recently, this result was generalized by Keszegh~\cite{Kes17+}, who showed that the intersection hypergraph of any family of
$n$ regions with a linear union complexity with respect to a family of pseudo-discs admits a CF-coloring with $O(\log n)$ colors.


\paragraph{String graphs.}
A string graph is an intersection graph of \emph{curves} (strings) in the plane. String graphs were introduced by Benzer~\cite{Ben59} 60 years ago in the context of genetic structures. In addition to their intrinsic theoretical interest, they have various applications (e.g., to electrical networks and to printed circuits).

String graphs were studied in numerous papers (see, e.g.,~\cite{EET76,FP10,FP12,KM91,Matousek14,PT02,SSS03}).
In particular, the chromatic number of string graphs was studied extensively. For example, it was shown that there are triangle-free string graphs with an arbitrarily large chromatic number \cite{PKKLMTW14}, while the more restricted class of so-called outer-string graphs are $\chi$-bounded (that is, their chromatic number is bounded by a function of their clique number, see, e.g.,~\cite{RW14}). In~\cite{CSS16} Chudnovsky et al. considered string graphs that are triangle-free but with an arbitrarily large chromatic number, and studied their induced subgraphs. Bounds on the chromatic numbers of some specific classes of string graphs, e.g., intersection graphs of frames (that are boundaries of axis-parallel rectangles), L-shapes, and grounded L-shapes (see definition below), were studied in~\cite{KW16,MCG96}.

\paragraph{Circle graphs.}
A circle graph is the intersection graph of $n$ chords on a circle; that is, its vertices are the chords and two vertices are adjacent if the corresponding chords intersect. Efficient algorithms for coloring circle graphs are important in VLSI designs. Specifically, in two-terminal switch box routing of VLSI physical designs, the wires connect pairs of terminals that lie on the boundary of a rectangle, which is the routing area. If two wires intersect, it is necessary to connect them on different layers so as to keep them electrically disconnected. Hence, the chromatic number of the circle graph whose vertices represent the wires is the minimal number of layers in which the wires can be arranged such that each pair of wires is electrically disconnected. See~\cite{Sherwani99} for more on this problem and its algorithmic aspects.


It is easy to see that an equivalent representation of the circle graph is the overlap graph of intervals on a line, in which the vertices are intervals on a fixed line and two vertices are adjacent if and only if the corresponding intervals intersect but neither of them contains the other.

The interval overlap graph is known to be more difficult to handle with respect to coloring problems than the classical intersection graph of intervals on a line. For example, while the latter is perfect (i.e., its chromatic number is equal to its clique number), the best bounds known for the maximal chromatic number $f(k)$ of an interval overlap graph with clique number $k$ are $\Omega(k \log k) \leq f(k) \leq O(2^k)$ (see~\cite{Ko97}); the exponential gap between the upper and lower bounds has remained open for more than $30$ years.

\subsection{Our results}

In this paper, we study CF-coloring of graphs, focusing on several classes of string graphs.

\subsubsection{CF-coloring of circle graphs and grounded L-shapes}

Our first result determines the CF-chromatic number of the class of circle graphs, up to a constant factor.
\begin{theorem}
\label{CFint}
Let $\famI(n)$ denote the maximum CF-chromatic number of a circle graph on at most $n$ vertices. Then
$\famI(n)= \Theta(\log n)$.
\end{theorem}
Our proof of Theorem~\ref{CFint} yields a polynomial time algorithm for constructing a CF-coloring with $O(\log n)$ colors. This lies in contrast with general graphs, where determining $\pn(G)$ was shown in~\cite{GR15} to be hard to approximate within a factor of $n^{\frac {1}{2}-\epsilon}$.

As the overlap graph of intervals on a line is equivalent to a circle graph, Theorem~\ref{CFint} demonstrates that the aforementioned difference in the coloring problem between the interval overlap graph and the interval intersection graph also exists for CF-coloring: While for any interval intersection graph $G$ one can easily prove that $\pn(G) \leq 4$ (see~\cite{KS18}), for the overlap graph of intervals, the CF-chromatic number (in the worst case) can be much larger, as mentioned above. Similarly, the problem of computing the CF-chromatic number is much more difficult (and even hard to approximate). The proof method is based on a ``distance layer division'' technique, which, to the best of our knowledge, was not used before in the context of CF-coloring.


\medskip

Our second result concerns the CF-chromatic number of grounded L-shapes. An \emph{L-shape} is a curve that consists of a vertical segment and a horizontal segment in the shape of the letter L. A family of \emph{grounded} L-shapes is a family of L-shapes in which the vertical segment of each L intersects a fixed horizontal line. It is easy to see that the class of intersection graphs of grounded L-shapes includes the class of circle graphs. Intersection graphs of L-shapes, and of grounded L-shapes, were studied, e.g., in~\cite{BT16,CU13,FKMU16}. For example, McGuinness~\cite{MG96} showed in 1996 that the intersection graph of grounded L-shapes is $\chi$-bounded.

While for the intersection graphs of general L-shapes, the CF-chromatic number can be as large as $\Theta(\sqrt n)$, in the case of grounded L-shapes, we obtain the following polylogarithmic bound:
\begin{theorem}
\label{thm:L-shapes}
Let $L(n)$ denote the maximum CF-chromatic number of a family of at most $n$ grounded L-shapes. Then $\Omega(\log n) \leq L(n) \leq O(\log^{3} n).$
\end{theorem}
\noindent As in the proof of Theorem~\ref{CFint}, the proof yields a polynomial-time coloring algorithm.


\subsubsection{$k$-CF coloring of string graphs}

As mentioned above, the CF-chromatic number of a string graph may be as large as $\Theta(\sqrt{n})$, which was shown by Cheilaris~\cite{CheilarisCUNYthesis2009} and by Pach and Tardos~\cite{CFPT09} to hold as an upper bound in any graph. We consider several classes of string graphs that are `hard to CF-color' in this sense, and we obtain polylogarithmic upper bounds on their $k$-CF chromatic numbers, for small values of $k>1$.
Specifically, we prove an upper bound of $\kpn(G) = O(\log n)$ for several classes of string graphs, for a constant $k$. In particular, we prove:

\begin{theorem}
\label{thm:framesLshapes}

For a class $\mathcal{F}$ and for $k \in \mathbb{N}$, denote by $\mathcal{F}^k(n)$ the largest $k$-CF chromatic number of an intersection graph of a family of $n$ elements of $\mathcal{F}$. Then:
\begin{enumerate}
\item For the class $\famL$ of L-shapes, $\famL^k(n)= \Theta(\log n)$ for all fixed $k \geq 2$.

\item For the class $\famF$ of frames, $\famF^k(n)= \Theta(\log n)$ for all fixed $k \geq 4$.
\end{enumerate}

\end{theorem}
We note that for both frames and L-shapes, a similar result could not be achieved for the CF-chromatic number; indeed, there exist graphs $G$ on $n$ frames (or L-shapes) with $\pn(G) = \Theta(\sqrt{n})$.


We obtain Theorem~\ref{thm:framesLshapes} as a special case of a new general result that allows bounding the $k$-CF chromatic number of string graphs that satisfy a complex-looking condition (Theorem~\ref{tlcs} below). We demonstrate the applicability of our general result by using it to prove the following seemingly unrelated upper bound on the CF-chromatic number of string graphs in terms of their chromatic number:
\begin{theorem}\label{thm:intro-string-graphs-UB}
Let $G=(V,E)$ be a string graph on a set $S$ of $n$ strings such that $\chi(G) \leq t$. Then $\pn(G) = O(t^2 \log n)$.
\end{theorem}
Using a general coloring result of Fox and Pach~\cite{FP14} for string graphs with clique number at most $\omega$, Theorem \ref{thm:intro-string-graphs-UB} implies that any string graph $G$ with clique number $\omega$ admits a CF-coloring with $O(\log^{c \omega} n)$ colors, where $c$ is a universal constant.



\subsubsection{$k$-CF coloring of general hypergraphs}

Clearly, there is in no sense comparing the above results on the $k$-CF chromatic number to the general upper bound $O(\sqrt n)$ on the CF-chromatic number of any graph. In order to understand what is the `worst case' of $k$-CF coloring, we obtain a near tight upper bound on the $k$-CF chromatic number for general hypergraphs. This result is of independent interest and generalizes the result of Cheilaris~\cite{CheilarisCUNYthesis2009} and of Pach and Tardos~\cite{CFPT09}, who showed that for any hypergraph $H$ with at most $m$ hyperedges, $\chi_{\mathrm{cf}}(H) = O(\sqrt{m})$.


\begin{theorem}\label{thm:intro-kcf-hypergraphs}
Let $H=(V,\E)$ be a hypergraph with $n$ vertices and at most $m$ hyperedges. For any $k > 1$,
$$\chi_{\text{\textup{k-cf}}}(H) = O(m^{\frac{1}{k+1}}\log^{\frac{k}{k+1}} n).$$
\end{theorem}
In particular, as the neighborhood hypergraph of a graph on $n$ vertices has at most $n$ hyperedges, this implies that any graph $G$ on $n$ vertices satisfies $\kpn(G)=O(n^{\frac{1}{k+1}}\log^{\frac{k}{k+1}} n)$.

Our proof technique is different from that of~\cite{CheilarisCUNYthesis2009,CFPT09}. We make use of the Lov\'{a}sz Local Lemma to find a \emph{$(k+1)$-weak} coloring of $H$ (i.e., a coloring of $H$ in which no hyperedge of size $\geq k+1$ is monochromatic), and then leverage it to a $k$-CF coloring of $H$ using an algorithm presented in~\cite{HS02}.

We prove that Theorem~\ref{thm:intro-kcf-hypergraphs} is tight up to logarithmic factors by presenting string graphs $G$ on $n$ vertices that satisfy $\kpn(G)=\Omega(n^{\frac{1}{k+1}})$.


\medskip Our results are summarized in Table~\ref{tab:results}.

\begin{table}[ht]
\begin{center}
\begin{tabular}{cccc}
\hline
Class of &  CF-chromatic & $k$-CF-chromatic & Source\\
 string graphs & number & number & \\
\hline
Chords of a circle & $\Theta(\log n)$ & $\Theta(\log n)$ for $k \geq 1$ & Thm.~\ref{CFint} and Lemma.~\ref{LBint}\\
L-shapes & $\Theta(\sqrt{n})$ (folklore)  & $\Theta(\log n)$ for $k \geq 2$ & Thm.~\ref{thm:framesLshapes}  \\
grounded L-shapes & $\Omega(\log {n})$, $O(\log^3 n)$  & $\Theta(\log n)$ for $k \geq 2$ & Thm.~\ref{thm:L-shapes} and Remark~\ref{rem:gLs}\\
Frames & $\Theta(\sqrt{n})$ (folklore) & $\Theta(\log n)$ for $k \geq 4$ & Thm.~\ref{thm:framesLshapes} \\
Chromatic number $t$ & $\Omega(\log n)$, $O(t^2 \log n)$ & $O(t^2 \log n)$ for $k \geq 1$ & Thm.~\ref{thm:intro-string-graphs-UB} and Prop.~\ref{Prop:Ext-thm-bounding-chromatic} \\
Clique number $\leq \omega$ & $\Omega(\log n)$, $O(\log^{c\omega} n)$ & $O(\log^{c\omega} n)$ for $k \geq 1$ & Thm.~\ref{thm:intro-string-graphs-UB} and Prop.~\ref{Prop:Ext-thm-bounding-chromatic} \\
\hline
\end{tabular}
\caption{For several classes of string graphs, the table lists the maximal ($k$)-CF-chromatic number of an intersection graph of $n$ strings that belong to the corresponding class.} \label{tab:results}
\end{center}
\end{table}

\paragraph{Organization of the paper.}
 After presenting some preliminaries in Section~\ref{sec:prelim}, we present the proofs of Theorems~\ref{CFint} and~\ref{thm:L-shapes} in Section~\ref{sec:main}.  In Section~\ref{sec:kcf}, the proofs of Theorems~\ref{thm:framesLshapes},~\ref{thm:intro-string-graphs-UB}, and~\ref{thm:intro-kcf-hypergraphs} are presented. 
 Finally, in Section~\ref{sec:discussion} we discuss the implications of our results.

\section{Preliminaries}
\label{sec:prelim}

In this section, we present several standard definitions and previous results that will be used throughout the paper.

\mn \textbf{Degree.} The degree of a vertex $v$ in a hypergraph $H=(V,\E)$ is the number of hyperedges that contain $v$. We denote the \emph{maximum degree} of $H$, i.e., the maximum degree of a vertex $v \in V(H)$, by $\Delta=\Delta(H)$.

\mn \textbf{Induced hypergraph.} An induced subhypergraph $H'=(V',\E')$ of a hypergraph $H=(V,\E)$ is a subhypergraph of $H$ in which $V' \subset V$, and the hyperedges in $\E'$ are the restriction of the hyperedges in $\E$ to $V'$; namely, $\E'=\{s\cap V' : s\in \E' \}$.

\mn \textbf{Proper coloring and $k$-weak coloring.} Let $H=(V,\E)$ be a hypergraph. We say that $c$ is a \emph{proper coloring} of $H$ if for every hyperedge $S \in \E$ with $\card{S}\geq 2$ there exist two vertices $u,v \in S$ such that $c(u) \neq c(v)$; that is, if every hyperedge with at least two vertices is non-monochromatic.

$c$ is called a \emph{$k$-weak coloring} if the same condition holds for any hyperedge of size $\geq k$. This notion was used implicitly in \cite{HS02,smoro} and then was explicitly defined and studied in the Ph.D. thesis of Keszegh \cite{Kes1,Kes2}. It is also related to the notion of
cover-decomposability and polychromatic colorings (see, e.g., \cite{gibsonvar,pach,PachToth}).


\mn \textbf{From weak coloring to $k$-CF coloring.} In \cite{HS02}, the authors prove the following lemma that connects weak coloring with $k$-CF coloring of hypergraphs.
\begin{lemma}
\label{lem:weakToKCF}
Let $k$ be a fixed constant. Let $n,t \in \mathbb{N}$, and let $H=(V,\E)$ be a hypergraph with $|V|=n$. If $H$, as well as any induced subhypergraph of $H$, admit a $(k+1)$-weak coloring with $t$ colors, then $$\chi_{\text{\textup{k-cf}}}(H)= O(t \log n).$$ Furthermore, a $k$-CF coloring with $O(t \log n)$ colors can be obtained by a polynomial time algorithm.
\end{lemma}
In the special case $k=1$, the lemma asserts that if $H$, as well as any induced subhypergraph of $H$, admit a proper coloring with $t$ colors, then $\chi_{\text{\textup{cf}}}(H)= O(t \log n)$.

For the sake of completeness, we present a proof sketch of the lemma.

\begin{proof} [Proof sketch]
We use the following algorithm, due to Har-Peled and Smorodinsky (\cite{HS02}).

\begin{algorithm}[htb!]
\caption{$k$-CFcolor$(H)$: {\it $k$-Conflict-Free-color a
hypergraph $H=(V,\E)$}.}
  \label{KCF-framework}
  \begin{algorithmic}[1]
    \STATE $i\gets 0$: {\it $i$ denotes an unused color}
    \WHILE{$V\neq \emptyset$}
    \STATE{\bf Increment:} $i\gets i+1$
    \STATE {\bf Auxiliary coloring:} {find a $(k+1)$-weak coloring $\chi$ of $H(V)$ with $t$ colors}
%
    \STATE {\bf $V' \gets$ Largest color class of $\chi$}
    \STATE {\bf Color:} $f(x)\gets i ~,~ \forall x\in V'$
    \STATE {\bf Prune:} $V\gets V\setminus V'$, $H\gets H(V)$
\ENDWHILE
\end{algorithmic}
\end{algorithm}
Algorithm~\ref{KCF-framework} outputs a valid $k$-CF-coloring of $H$, since for each $e \in \E$, the maximum color is assigned to at most $k$ vertices in $e$.


It remains to prove the claimed upper bound on the total number of colors used. Note that each time we execute Step 4 of Algorithm~\ref{KCF-framework}, we use a total of $t$ auxiliary colors, so in Step 7 we discard at least $\frac{\cardin{V}}{t}$ vertices. Thus, if $n$ is the number of vertices of $H$, then after the $i'th$ step there are at most $O(n(1-\frac{1}{t})^i)$ remaining vertices. Hence, the process terminates after $O(t \log n)$ steps. Since the total number of colors used by Algorithm~\ref{KCF-framework} equals the number of iterations, we obtain the asserted bound.
\end{proof}


\mn \textbf{The discrete interval hypergraph.} The discrete interval hypergraph, $I(n)$, is the hypergraph whose vertex set $S$ consists of $n$ points on a line $L$, and any interval on $L$ defines a hyperedge of all the points from $S$ that it contains. The following proposition is folklore (see, e.g.,~\cite{SMO10}).
\begin{proposition}
\label{prop:discrete}
$\cf (I(n)) = \lfloor \log n \rfloor +1$.
\end{proposition}

In the dual setting, given a multiset $\famI$ of $n$ intervals on some line $L$, the \emph{dual hypergraph of $\famI$} is the hypergraph with vertex set $\famI$  such that any point $x \in L$ defines the hyperedge $e_x= \{  s \in \famI | x \in s  \}$.  For our purposes, we need the proposition below in the case in which $\famI$ is a multiset; that is, an interval may appear in  $\famI$ multiple times, and each appearance gives rise to a vertex in the dual hypergraph of $\famI$.




\begin{proposition}
\label{prop:intdual}
The dual hypergraph of any finite multiset of intervals can be CF-colored (and, in particular, properly colored) with 3 colors.
\end{proposition}

Indeed, such a coloring can be obtained by considering a minimal cover of $\bigcup \famI$ by elements of $\famI$ (in which any point is covered by at most 2 intervals of $\famI$), coloring its elements alternately with 2 colors, and then coloring the remaining intervals with one additional color. (See, e.g.,~\cite{KS18} for details.)

\mn \textbf{$r$-degenerate graphs and planarity.} A graph $G$ is called \emph{$r$-degenerate} if each induced subgraph of $G$ contains a vertex of degree at most $r$.

The following easy claim relates the property of being $r$-degenerate to $(r+1)$-colorability.
\begin{claim}\label{Cl:r-degenerate}
Any $r$-degenerate graph is $(r+1)$-colorable.
\end{claim}

\mn We shall use Claim~\ref{Cl:r-degenerate} for graphs that can be represented as a union of planar graphs, via the following claim.
\begin{claim}\label{Cl:union-of-planar}
Let $G=(V,E)$ be a union of $s$ planar graphs on vertex set $V$ (that is, $G=(V,E_1 \cup E_2 \cup \ldots \cup E_s)$), where for each $i$, the graph $(V,E_i)$
is planar. Then $G$ is $6s$-colorable.
\end{claim}

\begin{proof}
As each $G_i=(V,E_i)$ is a planar graph, it follows from Euler's formula that $|E_i| \leq 3n-6$ for all $i$, where $n=|V|$. Hence,
$\cardin{E} \leq 3ns-6s$, and therefore, the average degree of $G$ is at most $6s-\frac{12s}{n}$. This implies that there exists a vertex in $V$ whose degree is at most $6s-1$. Since planar graphs are closed under vertex removal, the same holds for any induced subgraph of $G$, and thus, $G$ is $(6s-1)$-degenerate. Therefore, by Claim~\ref{Cl:r-degenerate}, $G$ is $6s$-colorable, as asserted.
\end{proof}

\mn \textbf{A planarity lemma.}
We shall use another lemma, whose proof is a standard planarity argument, of a type that appears many times in the context of string graphs (see, e.g., ~\cite{LW14,MG00,RW17}).
\begin{lemma}\label{lem:prelim-strings-planar}
Let $S_1$ and $S_2$ be two sets of pairwise disjoint strings (i.e., the strings in $S_1$ are pairwise disjoint, and the same is true for $S_2$). Assume that each string in $S_2$ crosses exactly two strings of $S_1$. Let $G=(S_1,E)$ be the graph whose edges are all pairs $a,b \in S_1$ that are crossed by the same string $x \in S_2$. Then $G$ is planar.
\end{lemma}
\noindent For the sake of completeness we provide the proof.

\begin{proof}
We present a planar drawing of the graph $G$. For any string $s \in S_1$, we choose a point $x_s \in s$. These points represent the vertices of $G$.
For each edge $e=(s_1,s'_1) \in E$, we define the planar drawing of $e$ as follows. Since $e \in E$, there exists $s_2 \in S_2$ that crosses
$s_1$ at the point $x_1$ and $s'_1$ at the point $x_2$ (see Figure~\ref{fig:string-planar}).

\begin{figure}[tb]
\begin{center}
\scalebox{0.6}{
\includegraphics{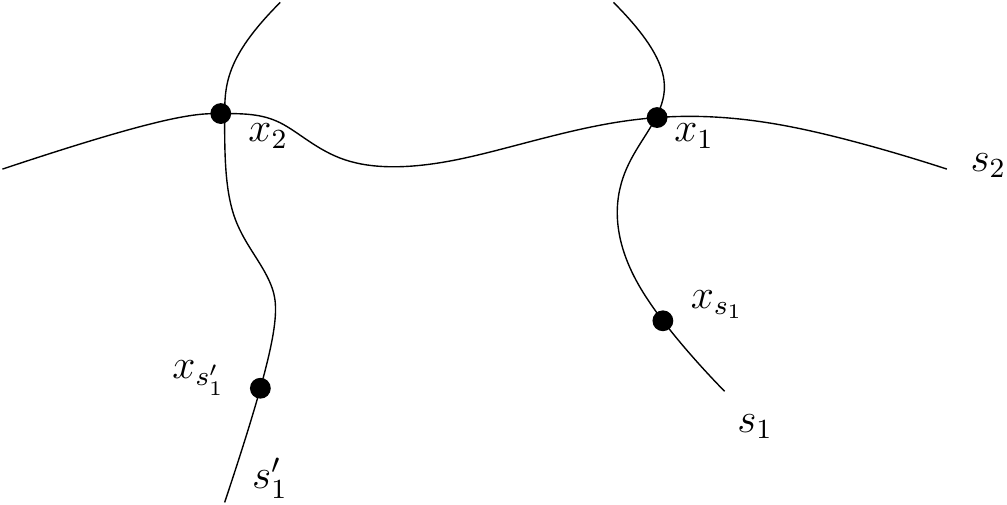}
} \caption{Illustration for the proof of Lemma~\ref{lem:prelim-strings-planar}} \label{fig:string-planar}
\end{center}
\end{figure}

The planar representation of the edge $(s_1,s'_1)$ consists of the union of three curves: The first one connects $x_{s_1}$ to $x_1$  on the curve $s_1$, the second one connects $x_1$ to $x_2$ on $s_2$, and the third one connects $x_2$ to $x_{s'_1}$ on $s'_1$. To avoid partially overlapping edges, we slightly perturb each edge without changing its endpoints, in a way that does not create new crossings.

To prove that $G$ is planar, it is sufficient to show that for any two vertex-disjoint edges $(s_1, s'_1), (t_1, t'_1) \in E$, the planar drawing of
$(s_1,s'_1)$ is disjoint from the planar drawing of $(t_1,t'_1)$. (The fact that this condition is sufficient follows, e.g., from the well-known
Hanani-Tutte Theorem~\cite{Cho34,Tut70}.)

Since the strings in $S_1$ are pairwise disjoint, the first and the third parts of the planar drawing of $(s_1,s'_1)$ are disjoint with the
first and the third parts of the planar drawing of $(t_1,t'_1)$. Similarly, the second part of the planar drawing of $(s_1,s'_1)$ is disjoint with the
second part of the planar drawing of $(t_1,t'_1)$, as the strings in $S_2$ are pairwise disjoint. Moreover, since we assumed that any string in $S_2$ crosses exactly two strings in $S_1$, the second part of the planar drawing of $(s_1,s'_1)$ is disjoint with the first and the third parts of $(t_1,t'_1)$ and vice versa. This completes the proof.
\end{proof}

\section{CF-coloring of circle graphs and their generalizations}
\label{sec:main}
\subsection{CF-coloring of circle graphs}
\label{sec:circle_graphs}


Combinatorial and coloring properties of circle graphs, also known as interval overlap graphs, have been extensively studied (see e.g. \cite{Gy85,Ko97,Ko98,Ko01,Wa15}).
As mentioned in the Introduction, the first results of Gy\'arf\'as' in \cite{Gy85} regarding circle graphs, already imply that, despite the apparent simplicity of interval overlap graphs, some questions have turned out to be really hard to answer.

The main goal of this subsection is to prove Theorem~\ref{CFint}. Let us recall its statement:

\mn \textbf{Theorem~\ref{CFint}}.
Let $\famI(n)$ denote the maximum CF-chromatic number of a circle graph on at most $n$ vertices. Then
$\famI(n)= \Theta(\log n)$.


\medskip

First, we prove the lower bound of Theorem~\ref{CFint} in a more general context:

\begin{lemma}
\label{LBint}
Let $\famI^k(n)$ denote the largest $k$-CF chromatic number of a circle graph on $n$ vertices for fixed $k \geq 1$. Then $\famI^k(n)= \Omega(\log n)$.
\end{lemma}

\begin{figure}[tb]
\label{IntF}
\begin{center}
\scalebox{0.6}{
\includegraphics{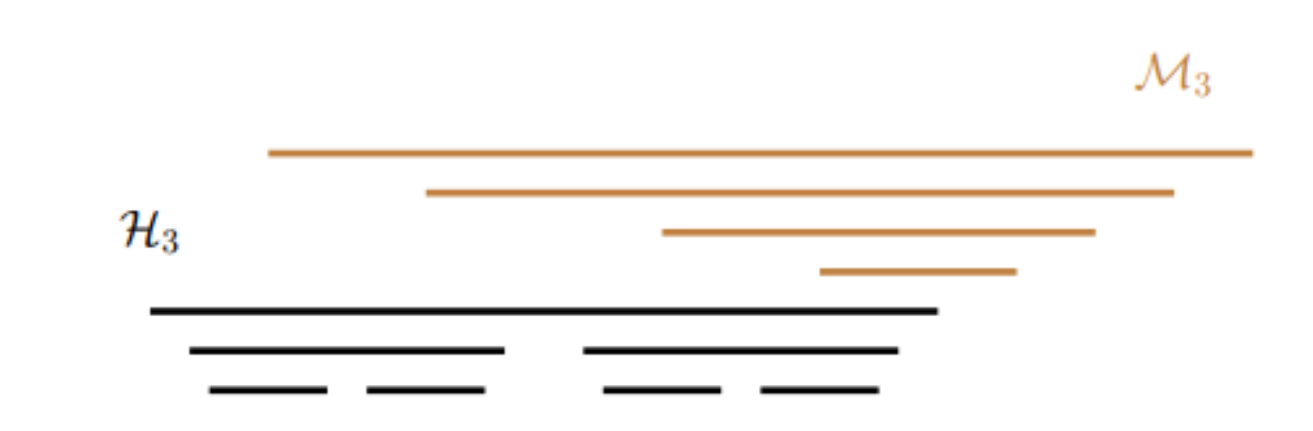}
} \caption{Representation of $\famF$ for $t=3$ and $k=1$.}
\end{center}
\end{figure}


\begin{proof}
For the sake of convenience, we use the representation of the circle graph as an interval overlap graph.
For $t\geq 1$ we construct a family $\famF$ of $n=2^{k(t-1)+1}+ 2^{k(t-1)}-1$ intervals such that any $k$-CF-coloring of the overlap graph of $\famF$  requires at least $t$ colors.  Thus, for a fixed $k$ we have $t=\Theta(\log n)$, and hence the CF-chromatic number of the overlap graph of $\famF$ is $\Omega(\log n)$.

The family  $\famF$ consists of two independent subfamilies $\famH_{k(t-1)+1}$  and  $\famM_{k(t-1)+1}$. The family $\famH_i$ is defined recursively. $\famH_1$ consists of an interval. To define $\famH_{i+1}$, we take $\famH_i$ and add to each minimal interval (w.r.t. containment) $I$ of $\famH_i$ two disjoint intervals, referred to as the \emph{left} and \emph{right direct descendents} of $I$, contained in $I$. The minimal intervals of $\famH_i$ are referred to as \emph{intervals of level $i$}, and the intervals of $\famH_{k(t-1)+1}$ strictly contained in some interval $I$ of $\famH_{k(t-1)+1}$ are referred to as \emph{descendants} of $I$. The family $\famM_{k(t-1)+1}$ consists of $2^{k(t-1)}$ intervals $I_1 \supseteq I_2 \ldots \supseteq I_{2^{k(t-1)}}$ whose right endpoints are to the right of intervals of $\famH_{k(t-1)+1}$. Furthermore, if $J_1,\ldots, J_{2^{k(t-1)}}$  are the minimal  intervals  of $\famH_{k(t-1)+1}$ ordered from left to right, then the left endpoint of $I_i$ is inside $J_i$; see Figure~2.

It is straightforward to check the cardinality of $\famF$. We now argue why $t$ colors are necessary. Consider any CF-coloring of $\mathcal{F}$. For $1\leq i\leq k(t-1)+1$, we will recursively define $\famD^i\subset \famM_{k(t-1)+1}$ of cardinality $i$ in which each  color appears at most $k$ times. This will prove the lemma, since at least $t$ colors must appear in $\famD^{k(t-1)+1}$.

Since the inductive proof in the sequel is quite cumbersome, we first present it informally, and then provide the formalization.


First, we consider the unique interval $Y=Y_1$ at level 1 of $\famH_{k(t-1)+1}$. $Y_1$ has a neighbor whose color appears at most $k$ times in the neighborhood of $Y_1$. Assume, for the sake of this example, this is $I_3$ and denote $\famD^1=\{I_3\}$.

Then we proceed to level 2 and consider the direct descendent of $Y_1$ that does not overlap with $I_3$. By our assumption this is the left direct descendent of $Y_1$, and we denote it $Y_2$. In addition, in its neighborhood there exists a neighbor whose color appears at most $k$ times; assume now that this is $I_1$ and thus denote $\famD^2=\famD^1 \cup \{I_1\}$.

In the third step, we proceed to level 3 and consider the direct descendent of $Y_2$ that does not overlap with $I_1$. By our assumption, this is the right direct descendent of $Y_2$, and we denote it $Y_3$. In its neighborhood there exists a neighbor whose color appears at most $k$ times; assume now that this is $I_2$ and thus denote $\famD^3=\famD^2 \cup \{I_2\}$.
We continue in this way $k(t-1)+1$ times.

When we finish, in any $\famD^i$ any color appears at most $k$ times. Indeed, assume, to the contrary, that there exists a color $C$ that appears more than $k$ times, and suppose that the first time we encountered it in the inductive process was in $I_{j_0}$. But note that for any $j>j_0$, $I_j \in \famD^i$ is also a neighbor of $Y_{j_0}$; thus, when during the inductive process we were in level $j_0$ and considered $Y_{j_0}$, the color $C$ appeared more than $k$ times in the neighborhood of $Y_{j_0}$, in contrast to the definition of $I_{j_0}$.


 We now present the formal proof. For inductive purposes, the unique interval $I$ of  $\famD^i\setminus\famD^{i-1}$ has an associated neighbor, denoted by $I_A$, of level $i$, which is not adjacent to any intervals of $\famD^{i-1}$ in the overlap graph of $\famF$. Furthermore, since by construction the associated neighbor of level $i+1$ will be contained in the one of level $i$, the associated neighbors form a chain according to the inclusion relation. $\famD^0= \emptyset$ and for $i=1$ we set $\famD^1$ to be any interval $W$ adjacent to the unique interval $Y$ of $\famH_1$  whose color appears at least once and at most $k$ times in $N(Y)$. This interval $Y$ is defined to be $W_A$. Assume that we defined $\famD^i$ for some $1\leq i\leq k(t-1)$ along with  associated neighbors of intervals contained there. Let $I$  be the unique interval of  $\famD^i\setminus\famD^{i-1}$. Define $J$ to be the unique direct descendent  of $I_A$, which does not contain the left endpoint of $I$. The interval $J$ is well defined, since only one of the $2$ disjoint direct descendants of $I_A$ contains the left endpoint of $I$. This actually means that $J$ does not overlap with  $I$, because the right endpoint of $I$ is also in the exterior of $J$. The latter is implied by definition of the family  $\famM_{k(t-1)+1}$ to which $I$ belongs. Since $I_A$ does not overlap with intervals of $\famD^{i-1}$, which is a basic property of associated neighbors that holds by induction, the interval $J\subset I_A$ does not overlap with them as well. Therefore, since $J$ also does not overlap with $I$,  it does not overlap with any interval of $\famD^i$. Let $V$ be a neighbor of $J$  whose color  $c$ appears at least once and at most $k$ times in $N(J)$. If we prove that the color $c$  appears among intervals of $\famD^i$ at most $k-1$ times, then defining $\famD^{i+1}=\famD^i\cup \{V\}$ and setting $V_A=J$ we complete the induction step.

Let us prove that $c$ appears at most $k$-$1$ times among the intervals of $\famD^i$. By contradiction, assume $c$  appears in $\famD^i$ at least $k$ times. Let $U\in \famD^i$ be the first interval of color $c$ included in  $\famD^i$. If $U'$ is another interval of $\famD^i$ also colored by  $c$, then $U'$ also overlap $U_A$. This is true, because $U'_A\subset U_A$, implying that the left endpoint of $U'$ is in $U_A$ while the right one is in the exterior of $U_A$. The latter is true for all intervals of the set $\famM_t$ to which $U'$ belongs. Thus, if we prove that $V$ overlaps with $U_A$ and is distinct from the other intervals of $\famD^i$, we obtain a contradiction to the assumption that the color $c$ appears at most $k$ times in the neighborhood of $U_A$. As showed above, there is no member of $\famD^i$ that overlaps with $J$; hence, in particular,  $J$ cannot overlap with intervals of $\famD^i$ whose color is $c$. Since $V$ overlaps with  $J$, obviously, $V\notin \famD^i$. Furthermore, by construction, $J$ is a descendent of $U_A$, because $J\subset I_A\subset U_A$. Thus, $V$ has its left endpoint in $U_A$ and its right endpoint outside $U_A$. Again, the latter is true for all intervals of the set $\famM_t$ to which $V$ belongs. This contradicts the fact that $c$ appears at most $k$ times in the neighborhood of $U_A$, so we can conclude.
\end{proof}

\begin{remark}
Note that Theorem~\ref{CFint} together with Lemma~\ref{LBint} imply for any fixed $k \geq 1$ that $\famI^k(n)= \Theta( \log n)$. Hence the problem of $k$-CF coloring of circle graphs is asymptotically solved.
\end{remark}

\begin{proof}[Proof of Theorem \ref{CFint}]
The lower bound follows from Lemma~\ref{LBint}. To prove the upper bound, it is more convenient to use the representation of the circle graph as the overlap graph of a family $\famI$ of $n$ intervals. By a small perturbation, we can assume that the intervals in $\famI$ have distinct endpoints.
In addition, we can assume that the graph is connected, as we can apply the argument presented below to each connected component of the graph separately.

We partition the family $\famI$ into layers that will be colored separately.
Set $I$ to be the interval of $\famI$ with the leftmost left endpoint. Let $\famS_i$ be the set of intervals of $\famI$ at distance exactly $i$ from $I$. (The distance, as well as all the neighborhoods mentioned in this section, is the shortest path distance defined w.r.t. the corresponding interval overlap graph, and not w.r.t the intersection graph.) Thus, $\famI$ is partitioned into disjoint `distance layer sets' $\famI=\bigcup_{i=0}^s\famS_i$.


For any $\tilde{I}\in\famS_i$, let $N(\tilde{I})$ denote the neighborhood of $\tilde{I}$, and let $N_{\famS_j}(\tilde{I})=N(\tilde{I}) \cap \famS_j$. Note that by  definition of the layer sets, for any $1 \leq i \leq s$, each $\tilde{I}\in\famS_i$ has a neighbor in $\famS_{i-1}$, and so $N_{\famS_{i-1}}(\tilde{I}) \neq \emptyset$ and $N_{\famS_{j}}(\tilde{I}) = \emptyset$ for all $j \neq i-1,i,i+1$.

Our proof consists of two stages:

\medskip \noindent \textbf{Stage~1: Satisfying each $\famS_{i}$ by coloring $\famS_{i-1}$.} We prove the following:
\begin{proposition}\label{Prop:Feeding-S_i}
For each $1 \leq i \leq s$, the intervals of $\famS_{i-1}$ can be colored with $O(\log n)$ colors such that each $\tilde{I}\in\famS_i$ has a neighbor in $N_{\famS_{i-1}}(\tilde{I})$  whose color appears in $N_{\famS_{i-1}}(\tilde{I})$ exactly once.
\end{proposition}
As the proof of the proposition is somewhat complex, we present it after the presentation of the second stage.

\medskip \noindent \textbf{Stage~2: CF-Coloring of the whole $\famS$ with $O(\log n)$ colors.} Naively, Stage~1 can be used to obtain a CF-coloring of the overlap graph of $\famS$ with $O(s \log n)$ colors by coloring each layer set with a different palette of colors, and adding one color to supply a uniquely colored neighbor to $I$. However, this is not an efficient bound, since $s$ might be large. We now show that $O(\log n)$ colors are sufficient.

Let $E_0$, $E_1$, and $E_2$ be three disjoint palettes of $O(\log n)$ colors and let $x$ be an additional color. We obtain a CF-coloring of the overlap graph of $\famI$ in the following way:
\begin{itemize}
\item[\textbf{i)}] We color the intervals of $\famS_s$ with an arbitrary color $c \in E_b$, where $b \equiv s(\mod 3)$.
\item[\textbf{ii)}] For each $1 \leq i \leq s$, we use Proposition~\ref{Prop:Feeding-S_i} to color the intervals of $\famS_{i-1}$ with the palette $E_j$, where $j \equiv i-1(\mod 3)$, such that each $\tilde{I}\in\famS_i$ has a uniquely colored neighbor in $N_{\famS_{i-1}}(\tilde{I})$.
\item[\textbf{iii)}] If after Step (ii), $I$ does not have a uniquely colored neighbor; we recolor one of its neighbors using one additional color $x$.
\end{itemize}
To see that this procedure indeed yields a CF-coloring of $\famI$, note that by the definition of the layers, for each $i \geq 1$, all the neighbors of each $\tilde{I} \in \famS_i$ belong to the sets $\famS_{i-1}, \famS_i$ and $\famS_{i+1}$. By Step (ii), $\tilde{I}$ has a uniquely colored neighbor in $\famS_{i-1}$. The sets $\famS_{i}$ and $\famS_{i+1}$ are colored using different palettes of colors, hence $\tilde{I}$ is guaranteed to have a uniquely colored neighbor. Step (iii) provides $I$ with a uniquely colored neighbor without destroying the uniquely colored neighbors of the other intervals, since the color $x$ is used in the coloring at most once. This completes the proof of the theorem.
\end{proof}

It is now left to present the proof of Proposition~\ref{Prop:Feeding-S_i}. We first present a skeleton of the proof and then fill in the required details.

\medskip \noindent \emph{Proof-skeleton of Proposition~\ref{Prop:Feeding-S_i}.}
Fix $1 \leq i \leq s$. Define $H=(V,\E)$ to be a hypergraph where $V=\famS_{i-1}$ and any $J \in \famS_i$ defines a hyperedge $e_J$ of all the intervals in $\famS_{i-1}$ that overlap $J$. To prove Proposition~\ref{Prop:Feeding-S_i}, we have to show that $H$ admits a CF-coloring with $O(\log n)$ colors. The proof consists of four steps.

\medskip \noindent \textbf{Step 1: Reduction to a proper coloring with a constant number of colors.} Recall that by Lemma~\ref{lem:weakToKCF}, to prove that $H$ admits a CF-coloring with $O(\log n)$ colors, it is sufficient to show that $H$, as well as any of its induced sub-hypergraphs, admit a proper coloring with a constant number of colors.

\medskip \noindent \textbf{Step 2: Transformation to L-shapes.} To obtain such a coloring, we first transform the elements of $\famI$ into grounded L-shapes, in a way that will help us to apply planarity arguments. We do so using the following slightly more general claim, which allows transforming any overlap graph of intervals into an intersection graph of L-shapes. (We also show in this claim that instead of L-shapes, one may obtain an intersection graph of frames. This will be useful in the sequel.)

\begin{claim}
\label{cl:olap_is_frames}
\begin{enumerate}
\item Any overlap graph of a family $\mathcal{W}$ of intervals is isomorphic to an intersection graph of a family $\mathcal{L}$ of grounded of L-shapes. Furthermore, there exists an isomorphism $\varphi: \mathcal{W} \rightarrow \mathcal{L}$ such that for any two overlapping intervals $[a,b]$ and $[c,d]$ where $a<c<b<d$, the vertical edge of $\varphi([c,d])$ intersects the horizontal edge of $\varphi([a,b])$. See figure 3.

\item Similarly, any interval overlap graph is isomorphic to an intersection graph of frames.
\end{enumerate}
\end{claim}

\begin{figure}
\label{Int_Lsh}
\centering
  \begin{tikzpicture}[scale=0.8]

  \draw[line width=2pt] (-1,0) -- (9,0);

  \draw[line width=2pt, brown] (0,0.4) -- (3,0.4);
  \draw[line width=2pt, brown] (0,0) -- (0,-1) -- (3,-1);

  \draw[line width=2pt, brown] (2,0.8) -- (8,0.8);
  \draw[line width=2pt, brown] (2,0) -- (2,-3) -- (8,-3);

  \draw[line width=2pt, brown] (5,0.4) -- (7,0.4);
  \draw[line width=2pt, brown] (5,0) -- (5,-2) -- (7,-2);

   \node[thick] at (1.5,0.7) (nodeA) { $W_1$\normalsize};
   \node[thick] at (6,-0.4) (nodeB)  { $W_2$\normalsize};
   \node[thick] at (5,1.2) (nodeC)   {$W_3$\normalsize};

   \node[thick] at (1.2,-1.4) (nodeX) { $\phi(W_1)$\normalsize};
   \node[thick] at (6,-2.4) (nodeY) { $\phi(W_2)$\normalsize};
   \node[thick] at (5,-3.4) (nodeZ) {$\phi(W_3)$\normalsize};

  \end{tikzpicture}
    \caption{A collection $\famW=\{W_1, W_2, W_3\}$ of intervals and the corresponding collection  of grounded L-shapes $\phi(\famW)=\{\phi(W_1), \phi(W_2), \phi(W_3)\}$.}
\end{figure}
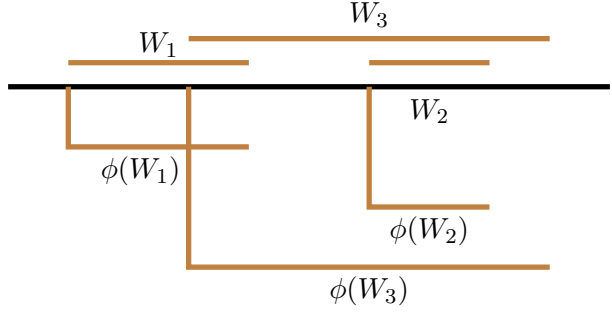

\begin{proof}[Proof of Claim~\ref{cl:olap_is_frames}.]
Let us sort the intervals  $W^1=[w^1_x,w^1_y], \ldots, W^r=[w^r_x,w^r_y]$ of  $\mathcal{W}$ according to their right endpoints, and turn each $W_i$ into an L-shape $\varphi(W_i)$, which is a union of the vertical segment $[(w^i_x,0),(w^i_x,-i)]$ and the horizontal segment $[(w^i_x,-i),(w^i_y,-i)]$.

If $W^i=[w^i_x,w^i_y]$ and $W^j=[w^j_x,w^j_y]$ overlap, then w.l.o.g. $w^i_x< w^j_x<  w^i_y<w^j_y$. This implies that the vertical part $[(w^j_x,0),(w^j_x,-j)]$ of $\varphi(W^j)$ intersects the horizontal part $[(w^i_x,-i),(w^i_y,-i)]$ of $\varphi(W^i)$, because $i <j$. If  $W^i$ and $W^j$ are nested or disjoint, a straightforward analysis implies the corresponding L-shapes $\varphi(W^i)$ and $\varphi(W^j)$ are disjoint.

Similarly, given a family  $\mathcal{W}$ of intervals, we turn each interval $I=[a,b] \in \mathcal{W}$ into a frame $[(a,-h(I)),(a,h(I))]\cup [(a,h(I)),(b,h(I))] \cup [b,h(I), (b,-h(I))]\cup [(b,-h(I)), (a,-h(I))]$, where $h(I)$ is the length of the largest chain according to the inclusion order whose maximal element is $I$. Clearly, the overlap graph of $\mathcal{W}$ is the intersection graph of the resulting family of frames. This completes the proof of Claim~\ref{cl:olap_is_frames}.
\end{proof}

We apply the transformation $\varphi$ to $\famS_{i-1} \cup \famS_{i}$ and denote by $\tilde{\famS}_{i-1}$ and $\tilde{\famS_i}$ the images of  $\famS_{i-1}$ and $\famS_i$, respectively. For each $W \in \famS_{i-1} \cup \famS_i$, we denote by $\tilde{W}$ the image of $W$, and by $h(\tilde{W})$ and $v(\tilde{W})$ the horizontal and the vertical parts of $\tilde{W}$, respectively.

\medskip \noindent \textbf{Step 3: Construction of three auxiliary graphs.} We define three auxiliary graphs $G_1,G_2,G_3$ on the vertex set $\famS_{i-1}$, in such a way that each hyperedge of $H$ of size $\geq 2$ is guaranteed to contain an edge of $G_1 \cup G_2 \cup G_3$. As a result, to prove that $H$ admits a proper coloring with a constant number of colors, it is sufficient to show that $\chi(G_1 \cup G_2 \cup G_3)$ is bounded by a constant.

The auxiliary graphs are defined as follows:
\begin{itemize}
\item For any $W_1,W_2 \in \famS_{i-1}$, $\{W_1,W_2\} \in E(G_1)$ iff there exists $\tilde{W} \in \tilde{\famS_{i}}$ such that $h(\tilde{W_1})$ and $h(\tilde{W_2})$ are consecutive along $v(\tilde{W})$; namely, such that no element of $\tilde{\famS}_{i-1}$ intersects $v(\tilde{W})$ between the points $\tilde{W} \cap \tilde{W_1}$ and $\tilde{W} \cap \tilde{W_2}$.
\item For any $W_1,W_2 \in \famS_{i-1}$, $\{W_1,W_2\} \in E(G_2)$ iff there exists $\tilde{W} \in \tilde{\famS_{i}}$ such that $v(\tilde{W_1})$ and $v(\tilde{W_2})$ are consecutive along $h(\tilde{W})$.
\item For any $W_1,W_2 \in \famS_{i-1}$, $\{W_1,W_2\} \in E(G_3)$ iff there exists $W \in \famS_{i}$ such that $N_{\famS_{i-1}}(W)=\{ W_1, W_2 \}$, $W_1$ is a left neighbor of $W$ (i.e., $W_1$ contains the left endpoint of $W$), and $W_2$ is a right neighbor of $W$.
\end{itemize}

Let $J \in \famS_i$, and let $e_J \in E(H)$ be the hyperedge of all the intervals in $\famS_{i-1}$ that overlap $J$. If $J$ has at least two left neighbors in $\famS_{i-1}$ then by the properties of the transformation $\varphi$, the set $\tilde{\famS}_{i-1}$ contains at least two L-shapes that intersect the vertical part of $\tilde{J}$. In particular, $\tilde{\famS}_{i-1}$ contains two such L-shapes that are \emph{consecutive} along $v(\tilde{J})$, and so $e_J$ contains an edge of the graph $G_1$. By a similar argument, if $J$ has at least two right neighbors in $\famS_{i-1}$, then $e_J$ contains an edge of $G_2$. Finally, if $J$ has exactly one left neighbor and exactly one right neighbor in $\famS_{i-1}$, then $e_J$ is an edge of $G_3$ (as this is how the edges of $G_3$ are defined). Hence, any hyperedge of $H$ of size $\geq 2$ does indeed contain an edge of $G=G_1 \cup G_2 \cup G_3$.

Furthermore, it is clear from the construction that to prove that some induced sub-hypergraph $\bar{H}$ admits a proper coloring with a constant number of colors, we can construct a corresponding auxiliary graph $\bar{G} = \bar{G_1} \cup \bar{G_2} \cup \bar{G_3}$, and show that $\chi(\bar{G})$ is bounded by a constant using the same argument as for $G$. Hence, we will be done once we show that $\chi(G)$ is bounded by a constant.

\medskip \noindent \textbf{Step 4: Proof that the auxiliary graph $G$ is 16-colorable.} This step consists of two lemmas whose proof is presented below.
\begin{lemma}\label{lem:G_1 and G_2 are planar}
The graphs $G_1$ and $G_2$ are planar.
\end{lemma}

\begin{lemma}\label{lem:max-degree-of-G_3}
The maximal degree of the graph $G_3$ satisfies $\Delta(G_3) \leq 4$.
\end{lemma}

The combination of Lemmas~\ref{lem:G_1 and G_2 are planar} and~\ref{lem:max-degree-of-G_3} implies immediately that $G$ is 16-colorable. Indeed, we show that it is $15$-degenerate. Let $G'$ be a subgraph of $G$ with, say, $i$ vertices. $G'$ is the union of graphs $G'_1 \subset G_1$, $G'_2 \subset G_2$, and $G'_3 \subset G_3$. By Euler's formula, the number of edges of $G'_1$ and of $G'_2$ is at most $3i-6$. Since the maximal degree of $G'_3$ is at most $4$, it has at most $2i$ edges. So altogether, the number of edges in $G'$ is at most $8i-12$, and thus, its average degree is at most $16-\frac{24}{i}$. Hence, there is at least one vertex with degree at most $15$.  By Claim~\ref{Cl:r-degenerate}, this implies that $G$ is 16-colorable.


This completes the skeleton of the proof of Proposition~\ref{Prop:Feeding-S_i}. It now remains only to present the proofs of Lemmas~\ref{lem:G_1 and G_2 are planar} and~\ref{lem:max-degree-of-G_3}.

\begin{proof}[Proof of Lemma~\ref{lem:G_1 and G_2 are planar}.]
We present the proof for $G_1$; the case of $G_2$ is identical. We derive the planarity of $G_1$ from Lemma~\ref{lem:prelim-strings-planar}. To this end, we set (in the notations of Lemma~\ref{lem:prelim-strings-planar}) $S_1=\{  h(\tilde{W}) | W \in \famS_{i-1} \}$. Then, for each $(W_1,W_2) \in E(G_1)$ we consider a segment of $v(\tilde{W})$ along which $h(\tilde{W_1})$ and $h(\tilde{W_2})$ are consecutive, and perform on these segments a small perturbation such that by setting $S_2$ (from Lemma~\ref{lem:prelim-strings-planar}) to be the obtained family of vertical segments, $S_1$ and $S_2$ satisfy the conditions of Lemma~\ref{lem:prelim-strings-planar}. It is clear that the graph whose planarity is asserted by Lemma~\ref{lem:prelim-strings-planar} is isomorphic to $G_1$.
\end{proof}

%
%

\noindent The proof of Lemma~\ref{lem:max-degree-of-G_3} is a bit more complex, and requires some preparations.

For each $1 \leq i \leq s$, the layer $\famS_i$ can be partitioned into subfamilies $\famF_1, \ldots, \famF_m$,  where the union of all intervals in each $\famF_j$  is an interval, and where the intervals $I_j=\bigcup_{J \in \famF_j}J$, $j=1,2,\ldots,m$, are pairwise disjoint. The following auxiliary claim is a variant of~\cite[Lemma~1]{Gy85}, although we use it in a different way and context.
\begin{claim}
\label{tzi}
For any interval $I_j$, each interval $T \in \famS_{i-1}$ has at most one endpoint in $I_j$.
\end{claim}

\begin{proof}[Proof of Claim~\ref{tzi}.]
By contradiction, assume that some $T \in \famS_{i-1}$ has both endpoints in $I_j$. Let $P_0=I, \ldots , P_{i-1} =T$ be a shortest path from $I$ to $T$ in the overlap graph on $\famI$. The definition of $I$ implies that its left endpoint is outside $I_j$, while $T$ has both endpoints inside $I_j$. Hence, there exists an index $\ell \leq i-2$ such that $P_\ell$ has one endpoint inside $I_j$ and one endpoint in the exterior of $I_j$. (Indeed, it is impossible that for two overlapping intervals $a,b$, one lies entirely inside some interval $c$ while both endpoints of the other are not included in $c$. W.l.o.g. we may assume that $\ell$ is the maximal index for which $P_\ell$ has one endpoint outside $I_j$.) This implies that $P_\ell$ overlaps some interval $T' \in \famS_i$, and thus, there is a path of length at most $i-1$ from $I$ to $T'$, contradicting the fact that $T'$ is at distance $i$ from $I$.
\end{proof}

We construct a subfamily $\famF'$ of $\famS_i$ by choosing, for each edge $(W_1,W_2) \in E(G_3)$, a single interval $W \in \famS_i$ such that $W_1$ is its unique left neighbor and $W_2$ is its unique right neighbor. It is clear that such a choice is possible, by the definition of $G_3$. To prove Lemma~\ref{lem:max-degree-of-G_3}, we need the following easy claim:
\begin{claim}\label{Cl:Aux-nested}
No two intervals in $\famF'$ are nested.
\end{claim}

\begin{proof}[Proof of Claim~\ref{Cl:Aux-nested}.]
Assume, to the contrary, that for some $W,W' \in \famF'$, we have $W \subset W'$. Let $J^L$ and $J^R$ be the left and the right neighbors of $W$ in $\famS_{i-1}$ (and so, $W$ is the interval of $\famF'$ that corresponds to the edge $(J^L,J^R) \in E(G_3)$). If $I_b$ is the interval among $I_1,\ldots, I_m$ that contains $W$ and $W'$, then by Claim~\ref{tzi}, both $J^L$ and $J^R$ must have an endpoint outside $I_b$, and in particular, outside $W'$. Hence, $W'$ overlaps $J^L$ and $J^R$, contradicting the fact that for the edge $(J^L,J^R) \in E(G_3)$, only one interval of $\famF'$ overlaps both $J^L$ and $J^R$.
\end{proof}

We are now ready to present the proof of Lemma~\ref{lem:max-degree-of-G_3}.

\begin{proof}[Proof of Lemma~\ref{lem:max-degree-of-G_3}.]
Assume, to the contrary, that there exists $U \in \famS_{i-1}$ with $deg_{G_3}(U) \geq 5$. Without loss of generality, at least three intervals of $\famF'$ are right neighbors of $U$.

Let $J_1=[a_1,b_1], J_2=[a_2,b_2],$ and $J_3=[a_3,b_3]$ be three right neighbors of $U$ in $\famF'$ and let $I_v$ be the interval among $I_1,\ldots, I_m$ containing the right endpoint of $U$ (and hence also containing the intervals $J_1$, $J_2$, and $J_3$). As by Claim~\ref{Cl:Aux-nested}, no two intervals of $\famF'$ are nested, we can assume w.l.o.g. that $a_1  < a_2 < a_3 <b_1 <b_2 <b_3$.

Let $J^R_2$ be the unique right neighbor of $J_2$ in $\famS_{i-1}$. Then the right endpoint of $J^R_2$ is outside of $I_v$ (and, in particular, is on the right of $b_3$), while its left endpoint is included in either $]a_2, a_3 [$ or $]a_3, b_2 [$. We show that both cases are impossible.

\medskip \noindent \emph{Case 1: The left endpoint of $J^R_2$ is in $]a_2, a_3 [$.} In this case, $J^R_2$ is a right neighbor of $J_1$, and hence, both $J_1$ and $J_2$ have $U$ as a left neighbor and $J^R_2$ as a right neighbor. This is a contradiction, since by the definition of $\famF'$, for the edge $(U,J^R_2) \in E(G_3)$, only one interval of $\famF'$ overlaps both $U$ and $J^R_2$.

\medskip \noindent \emph{Case 2: The left endpoint of $J^R_2$ is in $]a_3, b_2 [$.} In this case, $J^R_2$ is a right neighbor of $J_3$, and hence, both $J_2$ and $J_3$ have $U$ as a left neighbor and $J^R_2$ as a right neighbor. This is a contradiction, as in Case 1.

\medskip \noindent This completes the proof of the lemma, and thus also of Proposition~\ref{Prop:Feeding-S_i} and of Theorem~\ref{CFint}.
\end{proof}

\paragraph{Algorithmic aspect.} By following the proof of Theorem~\ref{CFint}, it is easy to see that a coloring with $O(\log n)$ colors can be constructed in polynomial time. Indeed, by the structure of Stages~1 and~2, it is sufficient to show that the coloring of each layer provided in Proposition~\ref{Prop:Feeding-S_i} can be constructed in polynomial time. Considering the four steps of Proposition~\ref{Prop:Feeding-S_i}, we see that each of them can indeed be performed in polynomial time: the reduction of Step~1 can be performed in polynomial time as asserted in Lemma~\ref{lem:weakToKCF}; the transformation to L-shapes of Step~2 requires only ordering of the intervals (as seen in the proof of Claim~\ref{cl:olap_is_frames}); and the construction of the graphs in Step~3 can clearly be performed in polynomial time as can the coloring of the auxiliary graph in Step~4, as any $d$-degenerate graph can be $(d+1)$-colored in polynomial time by a recursive process.

\subsection{CF-coloring of the intersection graph of grounded L-shapes}
\label{sec:groundedLshapes}


\medskip \noindent Given a fixed horizontal line  $H$, referred to as the \emph{baseline}, an L-shape is said to be \emph{grounded} if it is contained in the lower halfplane delimited by $H$ and  has exactly one endpoint on $H$. Such an endpoint of an L-shape $\li$  is referred to as a \emph{basepoint}, and is denoted by $b(\li)$. A family of L-shapes is said to be \emph{grounded} if each member is grounded. The neighbors of $\li$ based to the left of it are called  \emph{left neighbors}. In particular, a left neighbor of $\li$ intersects $v(\li)$ - the vertical part of $\li$. Analogously, the neighbors of $\li$ based to the right of it are called  \emph{right neighbors}. In particular, a right neighbor of $\li$ intersects $h(\li)$. We assume that the L-shapes of a given family have distinct basepoints. Given a collection $\famF$ of L-shapes, $v(\famF)$ and $h(\famF)$ denote the set of vertical parts and the set of horizontal parts of L-shapes of $\famF$, respectively.

In this subsection, we prove Theorem~\ref{thm:L-shapes}. We shall use the following technical lemma, which exploits the geometric structure of L-shapes.
\begin{proposition}
\label{DF}
Let $\tilde{\famX}$ be a set of $m$ grounded L-shapes. Let $\tilde{\famI}$ be a set of grounded L-shapes such that:
\begin{enumerate}
\item Each L-shape of $\tilde{\famI}$ has exactly one left and one right neighbor in $\tilde{\famX}$;

\item The elements of $\tilde{\famI}$ are pairwise disjoint;

\item No two L-shapes of $\tilde{\famI}$ have the same neighborhood.
\end{enumerate}
Then $|\tilde{\famI}|=O(m\log m)$.
\end{proposition}


\begin{proof}

Let us denote by $T(m)$ the maximum possible size of $\tilde{\famI}$ when $\tilde{\famX}$ has cardinality $m$. One observes that if  $m \leq 1$, then $T(m)=0$.  Let $\tilde{L}$ be a vertical line such that at most $m/2$ L-shapes of $\tilde{\famX}$ are strictly to the left of $\tilde{L}$ and  at most $m/2$ L-shapes of $\tilde{\famX}$ are strictly to the right of $\tilde{L}$. Let us denote by $\tilde{\famX}_L$, $\tilde{\famX}_M$, and $\tilde{\famX}_R$ the sets of L-shapes of $\tilde{\famX}$  strictly to the left of $\tilde{L}$, crossing $\tilde{L}$, and strictly to the right of $\tilde{L}$, respectively.


Let us define $\famI'\subset \tilde{\famI}$ to be the subset consisting of all L-shapes for which at most one neighbor is in $\tilde{\famX}_L$ and at most one neighbor is in $\tilde{\famX}_R$. We prove that $|\famI'| \leq 15m$. The L-shapes of $\tilde{\famI}\setminus \famI'$ have both neighbors in  $\tilde{\famX}_L$ or both neighbors in  $\tilde{\famX}_R$, therefore our bound on $|\famI'|$ will mean that $|\tilde{\famI}|= |\famI'| + |\tilde{\famI}\setminus \famI'|\leq 15m+2T(m/2)$. Taking $\tilde{\famI}$ of size $T(m)$,  we get $  T(m)\leq 15m+2T(m/2)$, and thus, $|\tilde{\famI}|= O(m \log m)$.


 We partition the L-shapes of $\famI'$ into $3$ sets. Next, we prove an upper bound on the size of each of these sets.
 \begin{itemize}
 \item $\famI_1$ consists of the L-shapes of $\famI'$ that have a neighbor in $\tilde{\famX}_L$  and a neighbor in $\tilde{\famX}_R$.

 \item $\famI_2$ consists of the L-shapes of $\famI'$ that have a neighbor in $\tilde{\famX}_M$ and a neighbor in $\tilde{\famX}_L\cup \tilde{\famX}_R $.


 \item  $\famI_3$ consists of the L-shapes of $\famI'$ that have both neighbors in $\tilde{\famX}_M$. (See Figure~\ref{fig:edge-L-shapes}.)
 \end{itemize}

%
%
%

  \begin{figure}
   \subfigure{
   \centering
        \includegraphics[width=0.35\textwidth]{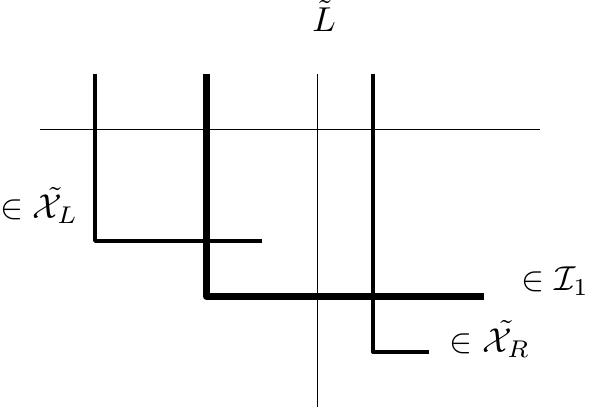}
    }
    \subfigure{
    \centering
        \includegraphics[width=0.35\textwidth]{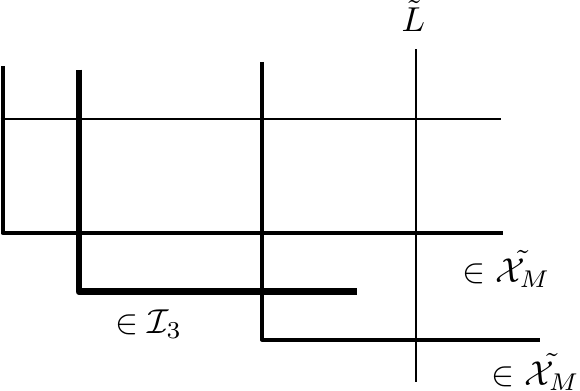}
    }
    \subfigure{
        \includegraphics[width=1.05\textwidth]{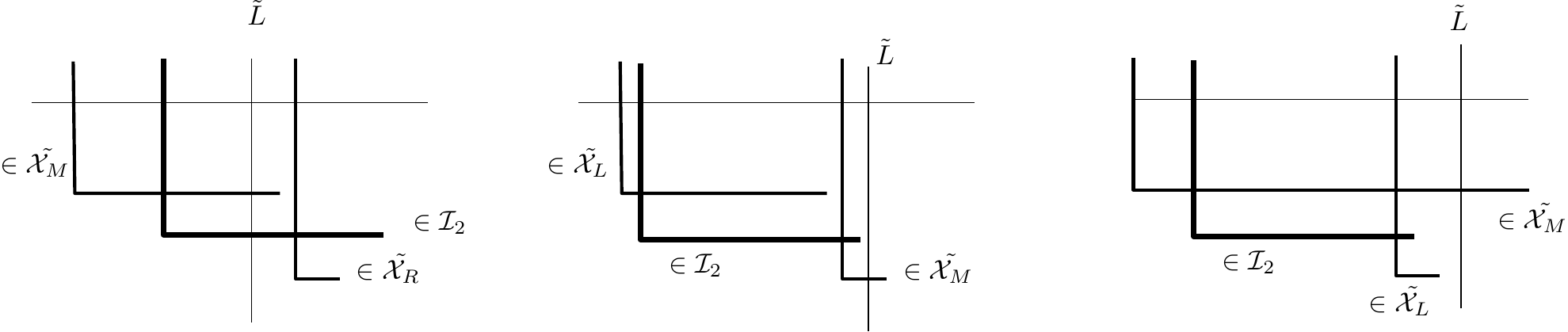}
    }
    \caption{Illustration for the definitions of $\famI_1,\famI_2$ and $\famI_3$.}
    \label{fig:edge-L-shapes}

    \end{figure}

\mn \textbf{Upper bound on the cardinality of $\famI_1$.} Each L-shape of $\famI_1$ intersects the horizontal part of some L-shape in $\tilde{\famX}_L$ and the vertical part of some other L-shape in $\tilde{\famX}_R$.  Define a graph whose vertex set is $h(\tilde{\famX}_L)\cup v(\tilde{\famX}_R)$  and whose edge set consists of pairs $\{h,v\}$ such that $h\in h(\tilde{\famX}_L),v\in v(\tilde{\famX}_R)$  and $h,v$ are intersected by some L-shape of $\famI_1$. By Lemma \ref{lem:prelim-strings-planar}, this graph is planar.
Indeed, setting $S_1=h(\tilde{\famX}_L)\cup v(\tilde{\famX}_R)$ and defining $S_2$ as the parts of $\famI_1$'s L-shapes between the two intersections, after a small perturbation, we are exactly in the setting of Lemma~\ref{lem:prelim-strings-planar}. (Note that the elements of $\famI_1$ are pairwise disjoint, and we take vertical segments from the right side of $\tilde{L}$ and horizontal segments from the left side of $\tilde{L}$.)
As there is a bijection between the edges of this graph and the L-shapes of  $\famI_1$, we get $|\famI_1|\leq 3 |\tilde{\famX}|= 3m$.

\medskip

To bound the cardinality of $\famI_2$ and $\famI_3$ we prove the following simple, yet powerful, claim.
 \begin{claim}
 \label{LPG}
Let  $\famA$, $\famB$ be two disjoint sets of grounded L-shapes and let $\famY$ be a set of pairwise disjoint grounded L-shapes such that:
\begin{enumerate}
\item Each $\li\in\famY$ intersects exactly one L-shape of  $\famA$ and exactly one L-shape of $\famB$;

\item No two L-shapes of $\famY$ have the same neighborhood.
\end{enumerate}
Then $|\famY|\leq 6(|\famA \cup\famB|)$.
 \end{claim}

\begin{proof}[Proof of Claim~\ref{LPG}.]
Let us partition $\famY$ into $\famY_1$ and $\famY_2$ as follows:

\begin{itemize}
\item $\famY_1$ consists of those L-shapes that intersect the horizontal part of an L-shape of $\famA$ and the vertical part of an L-shape of $\famB$.
\item $\famY_2$ consists of those L-shapes that intersect the vertical part of an L-shape of $\famA$ and the horizontal part of an L-shape of $\famB$.
 \end{itemize}


We prove that $|\famY_1| \leq 3|\famA \cup \famB|$; the proof that $|\famY_2| \leq 3|\famA \cup \famB|$ is analogous.
 Define $\tilde{G}$ to be the graph whose vertex set is $\famA \cup \famB$ and whose edge set consists of pairs $\{a,b\}$ with $ a \in \famA, b \in \famB$ such that $a$ and $b$ are intersected by some L-shape of $\famY_1$. We show that the graph $\tilde{G}$ is planar, which implies the claimed bound, since the number of edges in $\tilde{G}$ is exactly $|\famY_1|$ (recall that no two L-shapes of $\famY$ have the same neighborhood).

Let $\rho_{x_1}$ be the symmetry with respect to the $x_1$-axis. Set $v'(\famB)=\rho_{x_1}(v(\famB))$. Also define $\famY'=\{ S \cup \rho_{x_1}(S) : S \in \famY_1\}$. (See Figure~4.)

\begin{figure}[tb]
\label{fig:symmetry}
\begin{center}
\scalebox{0.6}{
\includegraphics{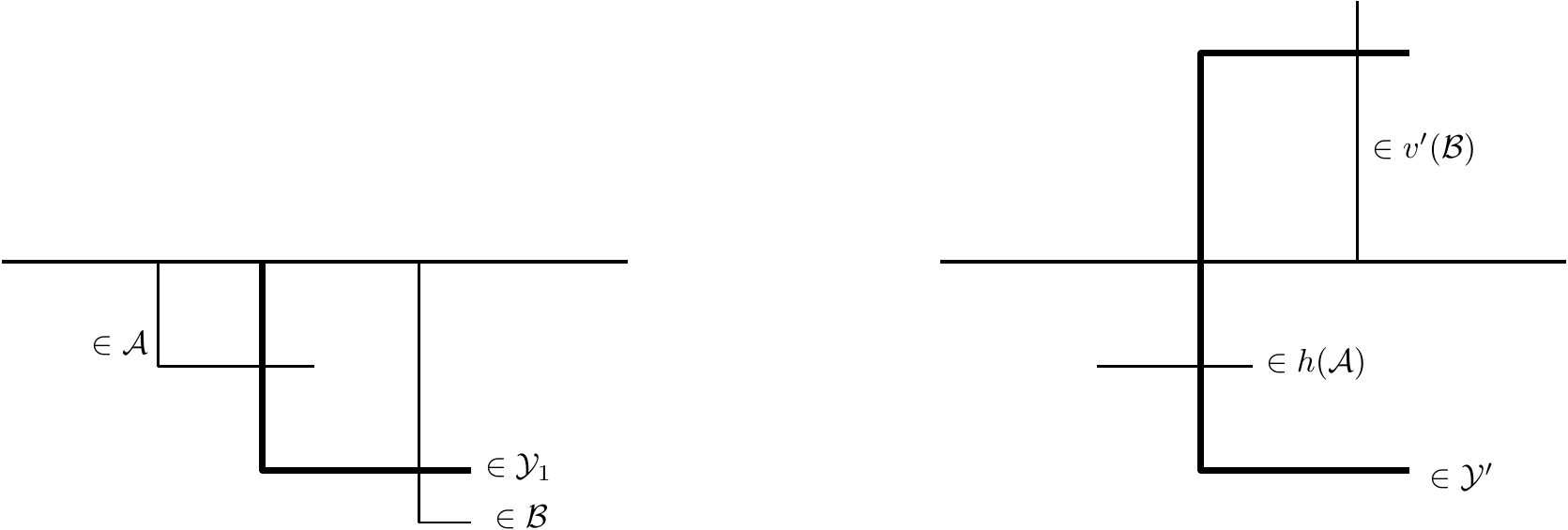}
} \caption{Illustration for the proof of Claim~\ref{LPG}.}
\end{center}
\end{figure}

 The graph  whose vertex set is $v'(\famB)\cup h(\famA)$ and whose edge set consists of pairs of vertices  intersected by some curve of $\famY'$ is isomorphic to $\tilde{G}$. Furthermore, this graph is easily seen to be planar. Indeed, on the one hand, the horizontal segments are pairwise disjoint and so are the vertical segments. On the other hand, no segment of $v'(\famB)$ can intersect a segment of $h(\famA)$, since these sets lie in different halfplanes delimited by the $x_1$-axis. Moreover, the curves of $\famY'$ are pairwise disjoint and each of them intersects exactly two segments of  $v'(\famB)\cup h(\famA)$. This means that putting $S_1=v'(\famB) \cup h(\famA)$ and $S_2=\famY'$, we are exactly in the setting of Lemma \ref{lem:prelim-strings-planar}.  Therefore, $\tilde(G)$ is planar and $|\famY_1|= |E(\tilde{G})|  \leq 3|A \cup B|$. This completes the proof of Claim~\ref{LPG}.
\end{proof}

\mn \textbf{Upper bound on the cardinality of $\famI_2$.} Defining $\famA=\tilde{\famX}_M$,  $\famB=\tilde{\famX}_L\cup \tilde{\famX}_R $, and $\famY=\famI_2$, we are exactly in the setting of Claim \ref{LPG}, which implies $|\famI_2|=|\famY| \leq 6|\famA \cup \famB| \leq 6 |\tilde{\famX}_M \cup \tilde{\famX}_L\cup \tilde{\famX}_R | = 6m$.


\mn \textbf{Upper bound on the cardinality of $\famI_3$.} Let us define a partial order $\prec$ on $\famX_M$ as follows. For $\li_1, \li_2 \in \famX_M$, we have  $\li_1 \prec \li_2 $ if and only if  $b(\li_1) < b(\li_2)$ and $\norm{v(\li_1)} < \norm{v(\li_2)}$.

 It is easy to see that if $\li\in\famI_3$ intersects $\li_1, \li_2 \in \famX_M$, then either $\li_1$ or $\li_2$ is minimal. Indeed, assume by contradiction that this is not the case. Observe that an L-shape of $\famI_3$ that intersects two L-shapes of $\famX_M$ must intersect the vertical part of one of them, and therefore, is based to the left of $\tilde{L}$. W.l.o.g., $\li$ intersects $h(\li_1)$ and $v(\li_2)$. Let $\li'_1\in \famX_M$ be such that $\li'_1 \prec \li_1$, which exists by the assumption that $\li_1$ is not minimal.  We get that $\li$ intersects both $h(\li_1)$ and $h(\li'_1)$, which contradicts the definition of $\tilde{\famI}$.

 On the other hand, it is not hard to see that $\li_1$ and $\li_2$  cannot be minimal simultaneously, since in such a case, any L-shape intersecting both $\li_1$ and $\li_2$,  must intersect either both $v(\li_1)$ and $v(\li_2)$ or both $h(\li_1)$ and $h(\li_2)$. This holds also for $\li$, which contradicts the assumption that $\li \in \famI_3$ has exactly one left neighbor and exactly one right neighbor.

To conclude, any L-shape of $\famI_3$ intersects exactly one minimal and exactly one non-minimal elements of $\famX_M$.

Defining $\famA$ to be the set of minimal elements of $\famX_M$, $\famB=\famX_M \setminus \famA$, and  $\famY=\famI_3$, we are exactly in the setting of Claim \ref{LPG}. Thus, $|\famI_3|= |\famY| \leq 6(|\famA|+ |\famX_M \setminus \famA|)= 6m$.




\medskip

Combining the upper bounds on $|\famI_1|$, $|\famI_2|$, and $|\famI_3|$, we obtain  $|\famI'|=|\famI_1|+|\famI_2|+|\famI_3| \leq 3m+6m+6m=15m$, and therefore, the inequality $T(m)\leq 15m+2T(m/2)$ holds as claimed.
This completes the proof of Proposition~\ref{DF}.
\end{proof} 

We are now ready to present the proof of Theorem~\ref{thm:L-shapes}. Let us recall its statement.

\mn \textbf{Theorem~\ref{thm:L-shapes}}.
Let $L(n)$ denote the largest CF-chromatic number of a family of at most $n$ grounded L-shapes. Then $\Omega(\log n) \leq L(n) \leq O(\log^{3} n).$

\begin{proof}[Proof-skeleton]
The lower bound is a straightforward consequence of Claim~\ref{cl:olap_is_frames} and of Lemma~\ref{LBint}.

For the upper bound, we first present the skeleton of the proof and then we fill in the required details.
Let $\famF$ be a family of $n$ grounded L-shapes. One can assume that no vertex of the underlying intersection graph is isolated, since such vertices can be colored with any color. We would like to find a partition of $\famF$ into $3$ families $\mathcal{A}_1,\mathcal{A}_2$, and $\mathcal{A}_3$  such that:
\begin{itemize}
\item[\textbf{i)}] $|\mathcal{A}_1|,|\mathcal{A}_3|\leq n/2,$  and no L-shape of  $\mathcal{A}_1$ intersects an L-shape of $\mathcal{A}_3$.
\item[\textbf{ii)}] Each $\mathcal{A}_i$ does not have any isolated vertex.
\item[\textbf{iii)}]$\mathcal{A}_2$ can be CF-colored with $O(\log^2 n)$ colors.
 \end{itemize}
If we find such a partition, then the upper bound of the theorem follows, since such a partition implies $L(n) \leq L(n/2)+O(\log^2 n)$, and thus, $L(n)=O(\log^3 n)$. Indeed, let us  CF-color $\mathcal{A}_1$ and $\mathcal{A}_3$ with the same palette of $L(n/2)$ colors and let us  CF-color $\mathcal{A}_2$ with a disjoint palette of $O(\log^2 n)$ colors. If $\li \in \mathcal{A}_1$, then it has a neighbor in $\mathcal{A}_1$ whose color $c$ is unique in $N_{\mathcal{A}_1}(\li)$ - the neighborhood of $\li$ restricted to $\mathcal{A}_1$. Furthermore, $c$ remains unique in $N_{\famF}(\li)$, since $\li$ does not intersect any L-shape of $\mathcal{A}_3$ and $c$ does not appear in $\mathcal{A}_2$. The case $\li \in \mathcal{A}_3$ is symmetric. If $\li \in \mathcal{A}_2$, then some color $\tilde{c}$ appears exactly once in $N_{\mathcal{A}_2}(\li)$. Furthermore, since $\tilde{c}$ is not used to color the families $\mathcal{A}_1$ and $\mathcal{A}_3$, the color $\tilde{c}$ appears exactly once in $N_{\famF}(\li)$. Thus, $\famF$ can be CF-colored with at most $L(n/2)+O(\log^2 n)$ colors. Since $\famF$ is an arbitrary family of $n$ grounded L-shapes, the inequality $L(n) \leq L(n/2)+O(\log^2 n)$ follows.

Let $L$ be a vertical line splitting the family $\mathcal{F}$ into $3$ families $\mathcal{F}_L$, $\mathcal{F}_M$, and $\mathcal{F}_R$ so that $\mathcal{F}_L$, $\mathcal{F}_R$ have cardinality at most $n/2$ and so that their elements lie strictly to the left of $L$ and strictly to the right of $L$, respectively, while $\mathcal{F}_M$ consists of those L-shapes of $\mathcal{F}$ that intersect $L$.

While it may seem natural to put in $\famA_2$ the elements of $\famF_M$, and continue recursively with $\famA_1$ and $\famA_3$ consisting of $\famF_L$'s and $\famF_R$'s elements, respectively, this trivial division does not work. Indeed, there may be elements of $\famF_M$ for which all neighbors are in $\famF_L \cup \famF_R$, and elements of $\famF_L \cup \famF_R$ for which all neighbors are in $\famF_M$. Hence, a more careful division into the families $\famA_i$ is needed here.

To present our division, we have to introduce some additional notations.
Let us partition $\mathcal{F}_M$ into two families: $\mathcal{F}^1_M$ and $\mathcal{F}^2_M$, where $\mathcal{F}^2_M$ consists of those L-shapes of $\mathcal{F}_M$ that do not have any neighbors in $\mathcal{F}_M$ (see Figure~5), and $\mathcal{F}^1_M = \mathcal{F}_M \setminus \mathcal{F}^2_M$.

\begin{figure}[tb]
\label{fig:L-shapes}
\begin{center}
\scalebox{0.6}{
\includegraphics{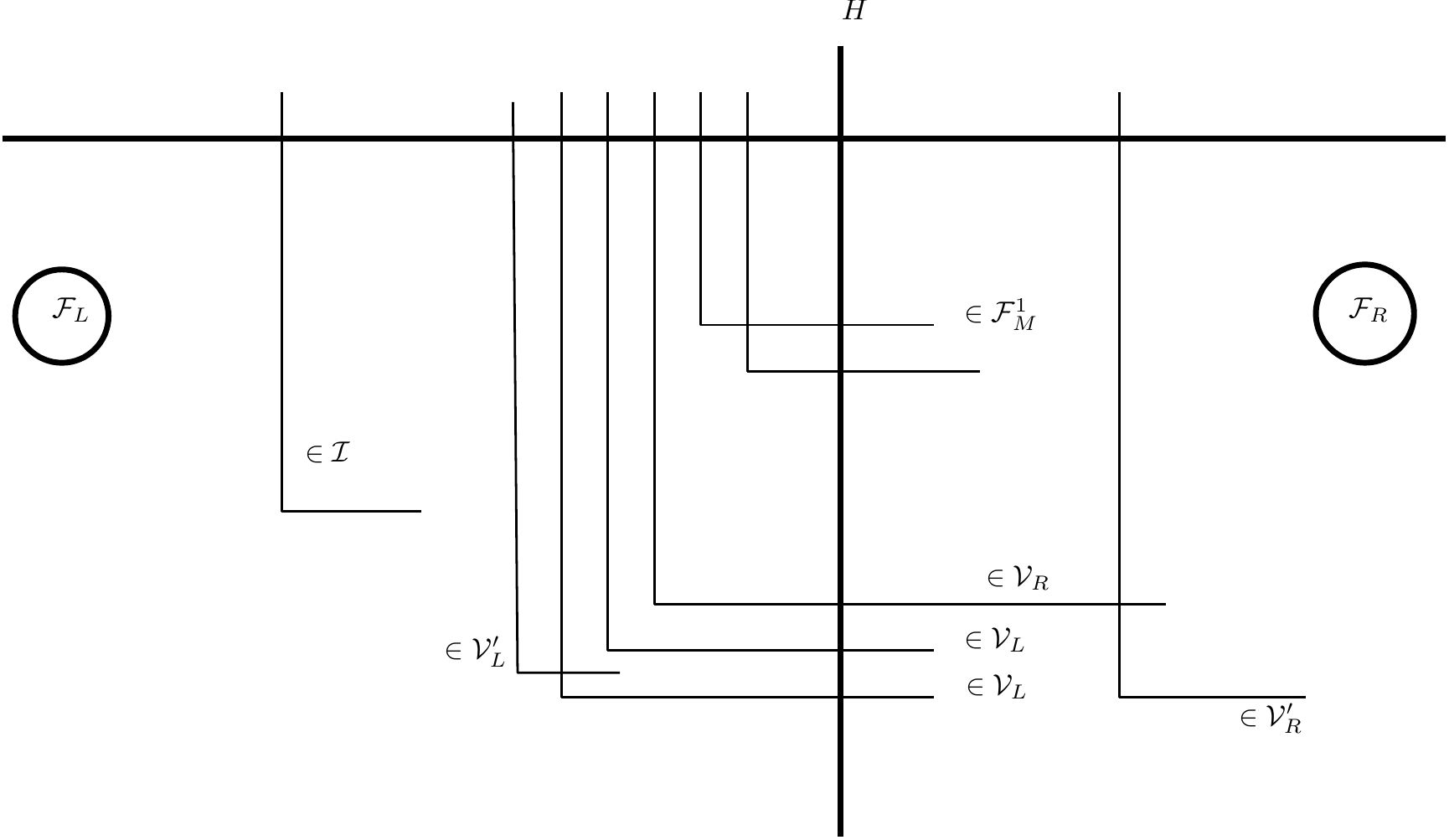}
} \caption{Illustration for the proof of Theorem~\ref{thm:L-shapes}.}
\end{center}
\end{figure}

Define by $\famV_L$ the subfamily of $\famF^2_M$ consisting of all L-shapes having a neighbor in $\famF_L$ and denote by $\famV'_L$ its set of neighbors in $\famF_L$. Similarly, define  by  $\famV_R$ the subfamily of $\famF^2_M$ consisting of all L-shapes having a neighbor in $\famF_R$ and denote by $\famV'_R$ its set of neighbors in $\famF_R$. One may observe that $\famV_L$ and $\famV_R$ are not necessarily disjoint; we will handle this issue later. By definition, an element of $\famF^2_M$ does not have any neighbors in $\famF_M$. Since each L-shape of  $\famF$  has at least one neighbor, this implies that any member of $\famF^2_M$ has a neighbor in $\famV'_L\cup \famV'_R$, and hence, $\famF^2_M=\famV_L\cup\famV_R$.

Let $\famI \subset (\famF_L \cup \famF_R) \setminus(\famV'_L \cup \famV'_R)$ consist of those L-shapes that do not have any neighbors in $(\famF_L\cup  \famF_R) \setminus (\famV'_L \cup \famV'_R)$. In particular, the elements of $\famI$ are pairwise disjoint. It is clear that the neighborhood of each L-shape of $\famI$ is in  $\famF^1_M \cup \famV'_L  \cup \famV'_R$. Indeed,  the neighbors of  $\famF^2_M $ are by definition in $\famV'_L  \cup \famV'_R$, which is contained in the complement of $\famI$.

We are now ready to present the division of $\famF$ into $\famA_1,\famA_2$, and $\famA_3$. We set
\[
\mathcal{A}_1=\famF_L\setminus (\famV'_L\cup \mathcal{I}), \qquad \famA_2=\famI\cup \famF_M \cup  \famV'_R \cup \famV'_L, \qquad \mathcal{A}_3=\famF_R\setminus (\famV'_R\cup \mathcal{I}).
\]
It is clear from the construction that $\mathcal{A}_1$ and $\mathcal{A}_3$ satisfy Conditions~(1,2) above. Hence, it is left to prove that $\mathcal{A}_2$ can be CF-colored with $O(\log^2 n)$ colors. We achieve the required coloring by constructing several intermediate colorings, as we detail below.

\mn \textbf{Notation.} \emph{A proper (resp. CF) coloring of a family} $\famK$ \emph{with respect to a family} $\famL$ is a proper (resp. CF) coloring of the hypergraph arising from $\famK \times \famL$, i.e., the hypergraph whose vertices are the elements of $\famK$, and any element of $\famL$ defines a hyperedge of all the elements of $\famK$ that it intersects.

\medskip \textbf{Intermediate colorings.} We shall use two intermediate colorings, $c_1$ and $c_2$, that we will combine into the required coloring of $\mathcal{A}_2$.

\mn \textbf{The coloring $c_1$.}  $c_1$ is a CF-coloring of $\famX=\famV'_L \cup \famF^1_M \cup \famV'_R$ with respect to $\famV_L \cup \famV_R \cup\famF^1_M \cup \famI$ with $O(\log^2 n)$ colors.

To obtain $c_1$, we define 4 proper colorings:
\begin{itemize}
\item[\textbf{$f_1$}] - A proper coloring of $\famX$ with respect to $\famI$ with $O(\log n)$ colors. This coloring is presented in Claim~\ref{cl:f_1}.
\item[\textbf{$f_2$}] - A proper coloring of $\famF^1_M$ with respect to $\famF^1_M$ with 49 colors. This coloring is presented in Claim~\ref{cl:f_2}.
\item[\textbf{$f_3$}] - A proper coloring of $\famV'_R$ with respect to $\famV_R$ with 4 colors. This coloring is presented in Claim~\ref{cl:f_3}.
\item[\textbf{$f_4$}] - A proper coloring of $\famV'_L$ with respect to $\famV_L$ with 3 colors. This coloring is presented in Claim~\ref{cl:f_4}.
\end{itemize}

We extend each of these four colorings to be defined on the whole set $\famX$, by adding one new color and using it to color the elements of $\famX$ on which the coloring $f_i$ is not defined. Then, we consider the Cartesian product $f_1 \times f_2 \times f_3 \times f_4$, in which any $\li \in \famX$ gets the `color' $(f_1(\li),f_2(\li),f_3(\li),f_4(\li))$. This product supplies a proper coloring of $\famX$ with respect to $\famV_L \cup \famV_R \cup\famF^1_M \cup \famI$ with $O(\log n)$ colors. It will be clear from the proofs that such a coloring can also be obtained  for any subset of $\famX$, and thus by Lemma~\ref{lem:weakToKCF} there exists a CF-coloring of $\famX$ with respect to $\famV_L \cup \famV_R \cup\famF^1_M \cup \famI$ with $O(\log^2 n)$ colors. (Note that a proper coloring of a hypergraph is exactly a $2$-weak coloring of it.)

\mn \textbf{The coloring $c_2$.}  $c_2$ is a CF-coloring of $\famF^2_M$ with respect to $\famV'_L \cup \famV'_R$ with $O(\log^2 n)$ colors, using a palette disjoint from the palette of $c_1$.

To obtain $c_2$, we define two CF-colorings:
\begin{itemize}
\item[\textbf{$f_5$}] - A CF-coloring of $\famF^2_M$ with respect to $\famV'_L$ with $O(\log n)$ colors.
\item[\textbf{$f_6$}] - A CF-coloring of $\famF^2_M$ with respect to $\famV'_R$ with $O(\log n)$ colors.
\end{itemize}
The asserted coloring $c_2$, obtained by the Cartesian product $f_5 \times f_6$, is presented in Claim~\ref{cl:c_2}.

To obtain these colorings, the most complex part is obtaining the coloring $f_1$, and this is where Proposition~\ref{DF} is used.

\textbf{Definition of a CF-coloring $c$ with $O(\log^2 n)$ colors on $\famA_2$.}
The coloring $c$ is defined in the following way:
\begin{itemize}
 \item[\textbf{i)}] $\famX=\famV'_L \cup \famF^1_M \cup \famV'_R$ is colored with $c_1$, 
 \item[\textbf{ii)}]$\famF^2_M $ is colored with $c_2$,
 \item[\textbf{iii)}]$\famI $ is colored with any color that does not appear in the palettes of $c_1$ and $c_2$.
\end{itemize}

By using an elementary case analysis, it is not hard to see that the coloring $c$ is, indeed, a CF-coloring of $\famA_2$ with $O(\log^2 n)$ colors.
Thus, the division of $ \famF $ into sets  $\famA_i$ satisfies all three properties mentioned at the beginning of the proof. This completes the proof-skeleton of Theorem~\ref{thm:L-shapes}. It now remains only to present the colorings $f_1-f_6$.
\end{proof}



In the rest of this section, we concentrate on presenting the six intermediate colorings. We start with the most interesting one, $f_1$.

\begin{claim}
\label{cl:f_1}
There exists a proper coloring $f_1$ of $\famX$ with respect to $\famI$ with $O(\log n)$ colors.
\end{claim}

\begin{proof}
We recall that $\famX=\famV'_L \cup \famF^1_M \cup \famV'_R$ and let $H_1$ be the hypergraph arising from  $\famX\times \famI$. We prove that there exists a proper coloring of $H_1$ with $O(\log n)$ colors. It will be clear that the same argument is also valid for the hypergraph arising from $\famX'\times \famI$ for any $X' \subset X$ of size $n'$, yielding a coloring with $O(\log n')$ colors. Analogously to the proof of step 3 in Proposition \ref{Prop:Feeding-S_i}, we define  $3$ graphs  $G_1, G_2, G_3$ as follows:

\begin{itemize}
\item For any $W_1,W_2 \in \famX$, $\{W_1,W_2\} \in E(G_1)$ iff there exists $W \in  \famI$ such that $h(W_1)$ and $h(W_2)$ are consecutive along $v(W)$; namely, such that no element of $\famX$ intersects $v(W)$ between the points $W \cap W_1$ and $W \cap W_2$.
\item For any $W_1,W_2 \in  \famX$, $\{W_1,W_2\} \in E(G_2)$ iff there exists $W \in  \famI$ such that $v(W_1)$ and $v(W_2)$ are consecutive along $h(W)$.
\item For any $W_1,W_2 \in \famX$, $\{W_1,W_2\} \in E(G_3)$ iff there exists $W \in  \famI$ such that $N_{\famX}(W)=\{ W_1, W_2 \}$, one of them (w.l.o.g., $W_1$) is a left neighbor of $W$ (i.e., $h(W_1)$ intersects $v(W)$), and the other is a right neighbor of $W$.
\end{itemize}

Clearly, any hyperedge of $H_1$ of size at least $2$ contains an edge of $G=G_1\cup G_2\cup G_3$. Therefore, a proper coloring of $G$ is a proper coloring of $H_1$.


By Proposition~\ref{DF}, any induced subgraph of  $G_3$ on $\tilde{n}$ vertices has $O(\tilde{n} \log \tilde{n})$ edges. Indeed, for each edge $e$ of $G_3$  select exactly one L-shape $\li_e\in \famI$ such that $N_{\famX}(\li_e)=e$. Since the elements of $\famI$ are, by definition, pairwise disjoint, if $\tilde{\famI}$ denotes the set of selected L-shapes and $\tilde{\famX}=\famX$, then $\tilde{\famX}$, $\tilde{\famI}$ satisfy the assumptions of Proposition~\ref{DF} and therefore, $|E(G_3)|=|\tilde{\famI}|= O(|\famX|\log |\famX|)$. The same reasoning is valid for induced subgraphs of  $G_3$.

Exactly as in Lemma \ref{lem:G_1 and G_2 are planar},  the graphs $G_1$ and $G_2$ are planar. Since planar graphs have a linear number of edges in terms of their number of vertices and are closed under taking induced subgraphs, this implies that any induced subgraph of  $G=G_1\cup G_2 \cup G_3$  on $\tilde{n}$ vertices has $O(\tilde{n} \log \tilde{n})$ edges. Therefore, $G$, a graph on $|\famX|$ vertices, is  $O(\log |\famX|)$-degenerate, and hence is  $O(\log |\famX|)$-colorable according to Claim \ref{Cl:r-degenerate}. Since any proper coloring of $G$ is a proper coloring of $H_1$, the assertion of Claim~\ref{cl:f_1} follows.
\end{proof} 

To obtain the coloring $f_2$, needed for the proof of Theorem~\ref{thm:L-shapes}, we use the following observation:

\begin{observation}
\label{LIO}
Let $\famG$ be a family of grounded L-shapes that intersect a vertical line $L$. Then the intersection graph of $\famG$ is an overlap graph of some family of intervals.
\end{observation}

 \begin{proof}
 Let $O$ be the intersection point of  $L$  and the baseline. For each $\li\in \famG$, define  $x(\li)$ and $y(\li)$ to be the distance between $O$ and  $b(\li)$ and the distance between $O$ and $\li\cap L$, respectively.  We map each $\li\in \famG$  to the interval $[-x(\li),y(\li)]$. Clearly, for each pair $\li, \li' \in\famG$, $\li$ intersects  $\li'$  if and only if $[-x(\li),y(\li)]$ and $[-x(\li'),y(\li')]$ overlap.
 \end{proof}

 We are now ready to present the coloring $f_2$.
 \begin{claim}
 \label{cl:f_2}
 There exists a proper coloring $f_2$ of $\famF^1_M$ with respect to $\famF^1_M$ with 49 colors.
 \end{claim}

 \begin{proof}

The intersection graph of $\famF^1_M$ is isomorphic to an overlap graph of intervals according to Observation \ref{LIO}. Hence, by the proof of Theorem~\ref{CFint}, it can be properly colored with a constant number of colors. Specifically, if we let $\famS_0,\famS_1, \ldots, \famS_d$ be the partition of $\famF^1_M$ into distance layers as in the proof of Theorem~\ref{CFint}, then by Steps~3,4 of the proof of Proposition~\ref{Prop:Feeding-S_i}, for any $2 \leq i \leq d$, we can properly color $\famS_{i-1}$ with respect to $\famS_i$ with 16 colors. Then we can use these colorings to obtain a proper coloring of the hypergraph $H_2$ arising from $\famF^1_M \times \famF^1_M$ with $3 \cdot 16+1=49$ colors, by the argument of Stage~2 in the proof of Theorem~\ref{CFint}. It is clear that the same argument holds for any induced sub-hypergraph of $H_2$.
 \end{proof} 

Next, we present the coloring $f_3$.
\begin{claim}
\label{cl:f_3}
 There exists a proper coloring $f_3$ of $\famV'_R$ with respect to $\famV_R$ with 4 colors.
 \end{claim}

 \begin{proof}
 We construct the coloring by using Lemma~\ref{lem:prelim-strings-planar}. We observe that in any intersection of $\li \in \famV_R$ with $\li' \in \famV'_R$, $h(\li)$ intersects $v(\li')$. In the notations of Lemma~\ref{lem:prelim-strings-planar}, we define $S_1=v(\famV'_R)$. To define $S_2$, for any $\li \in \famV_R$ that intersects at least two members of $\famV'_R$, we take some segment of $h(\li)$ between two such consecutive intersections, and add it to $S_2$. After a small perturbation of the constructed sets $S_1,S_2$, we can apply Lemma~\ref{lem:prelim-strings-planar} to assert that the graph arising from $S_1 \times S_2$ is planar, and thus, 4-colorable. It is easy to see that the obtained 4-coloring is a proper 4-coloring of the hypergraph arising from $\famV'_R \times \famV_R$. Furthermore, it is clear that the same also holds for any induced subhypergraph of it.
 \end{proof} 

 The following observation is used to obtain the colorings $f_4$ and $f_6$ below.
 \begin{observation}
 \label{obs:order}
 The L-shapes of $\famF^2_M$ naturally define a total order relation $<$, where  $\li_1 < \li_2$ if and only if the area delimited by $\li_1$, $L$ and the baseline is strictly contained in the area delimited by $\li_2$, $L$ and the baseline. (Recall that the L-shapes of $\famF^2_M$ are pairwise disjoint.) Moreover, if $\li_1 <\li_2 <\li_3$ and an L-shape of $\famV'_L$ intersects $\li_1$ and $\li_3$, it must also intersect $\li_2$.
 \end{observation}

Among the four proper colorings needed for the CF-coloring $c_1$, the last one to be presented is $f_4$.

\begin{claim}
\label{cl:f_4}
 There exists a proper coloring $f_4$ of $\famV'_L$ with respect to $\famV_L$ with 3 colors.
 \end{claim}

 \begin{proof}
The order relation defined in Observation~\ref{obs:order} on $\famF^2_M$ induces an order relation $\prec$ on $\famV_L$. Let $\ell_1\prec\ldots \prec \ell_s$ be the L-shapes of $\famV_L$. According to Observation~\ref{obs:order} each $\ell\in\famV'_L$ can be mapped to some $[a_1, a_2]$ such that $\ell$ intersects exactly $\ell_{a_1}\prec \ldots \prec \ell_{a_2}$ among the L-shapes of $\famV_L$. Let us call $\famJ$ the resulting multiset of intervals (we distinguish 2 intervals arising from different L-shapes of $\famV'_L$). Clearly, the hypergraph $H$ arising from $\famV'_L \times \famV_L$ is isomorphic to some \emph{spanning subhypergraph} of the dual hypergraph of $\famJ$.  (Note that, in this context, we use the notion of a \emph{spanning subhypergraph} for a hypergraph whose vertex set is the same as that of the original hypergraph and that contains some of the hyperedges of the original hypergraph.) Thus by Proposition~\ref{prop:intdual}, $H$ admits a proper $3$-coloring. It is easy to see that a similar proper $3$-coloring can also be obtained for any subset of $\famV'_L$.


\end{proof}

 In Claims~\ref{cl:f_1},~\ref{cl:f_2},~\ref{cl:f_3}, and~\ref{cl:f_4} we obtained the four proper colorings $f_1-f_4$ and thus proved the existence of the CF-coloring $c_1$ that was defined in the proof-skeleton of Theorem~\ref{thm:L-shapes}. In the following claim we obtain the CF-coloring $c_2$, by passing through two additional temporary CF-colorings, $f_5$ and $f_6$.
 \begin{claim}
\label{cl:c_2}
 There exists a CF-coloring $c_2$ of $\famF^2_M$ with respect to $\famV'_L \cup \famV'_R$ with $O(\log ^2 n)$ colors.
 \end{claim}

 \begin{proof}
 As in the proof of Claim~\ref{cl:f_3}, we use Lemma~\ref{lem:prelim-strings-planar}. Let $H_1$ be the hypergraph arising from $\famF^2_M \times \famV'_R$. Recall again that in any intersection of $\li \in \famF^2_M$ with $\li' \in \famV'_R$, $h(\li)$ intersects $v(\li')$. In the notations of Lemma~\ref{lem:prelim-strings-planar}, we let $S_1=h(\famF^2_M)$, and for any $\li' \in \famV'_R$ that intersects at least two elements of $\famF^2_M$, we take a segment of $v(\li')$ between two such consecutive intersections and add it to $S_2$. Exactly as in the proof of Claim~\ref{cl:f_3}, we can obtain a proper $4$-coloring of any induced subhypergraph of $H_1$. By Lemma~\ref{lem:weakToKCF}, this implies that there exists a CF-coloring $f_5$ of $H_1$ with $O(\log n)$ colors.

 As for the hypergraph $H_2$ arising from $\famF^2_M \times \famV'_L$, we use Observation~\ref{obs:order} and get that $H_2$ is isomorphic to the discrete interval hypergraph on $|\famF^2_M|$ vertices (defined in Section~\ref{sec:prelim}). Hence, by Proposition~\ref{prop:discrete}, there exists a CF-coloring $f_6$ of $H_2$ with $O(\log n)$ colors.

 Let $c_2=f_5 \times f_6$ be the Cartesian product of $f_5$ and $f_6$. Then $c_2$ is a CF-coloring of $\famF^2_M $ with respect to $\famV'_L \cup \famV'_R$ with $O(\log ^2 n)$ colors.
 \end{proof}

\section{$k$-CF Coloring of String Graphs and General Hypergraphs}
\label{sec:kcf}

The CF-chromatic number of many classes of string graphs approaches the general upper bound of $O(\sqrt{n})$. Indeed, take, for example, $m$ pairwise intersecting frames such that no three of them intersect at the same point. For any intersection point of a pair of frames, add a tiny frame that intersects only that pair. Since any two original frames are the neighborhood of some tiny frame, they must get different colors in any CF-coloring of the corresponding intersection graph. Thus we obtain a graph with $n=m+\binom{m}{2}$ vertices whose CF-chromatic number is $m=\Theta(\sqrt n)$. A similar construction holds, of course, for L-shapes, circles, etc.

In this section, we obtain polylogarithmic bounds on the $k$-CF-chromatic numbers of several classes of string graphs whose CF-chromatic number approaches $O(\sqrt{n})$, such as intersection graphs of frames and L-shapes. Before considering specific classes, we obtain in Section~\ref{sec:general} a tight bound on the $k$-CF-chromatic number of arbitrary hypergraphs (and thus, of arbitrary graphs as well). Our bound generalizes a result of Pach and Tardos~\cite{CFPT09}, who showed that for any hypergraph $H$ with at most $m$ hyperedges, $\chi_{\mathrm{cf}}(H) = O(\sqrt{m})$. In addition to its intrinsic interest, our result serves as a benchmark for the upper bounds we obtain for specific classes of string graphs in the rest of the section.

We then prove Theorem~\ref{tlcs} -- a general method that allows bounding the $k$-CF-chromatic numbers of various classes of string graphs. We use it to prove Theorem~\ref{thm:framesLshapes}, which asserts a tight bound of $O(\log n)$ on the $k$-CF-chromatic numbers of the intersection graphs of frames (for all $k \geq 4$) and of L-shapes (for all $k \geq 2$). As another application of Theorem~\ref{tlcs}, we prove Theorem~\ref{thm:intro-string-graphs-UB}, which allows bounding the CF-chromatic number of any string graph in terms of its chromatic number.

\subsection{$k$-Conflict-Free coloring for arbitrary hypergraphs}
\label{sec:general}

In this section, we prove Theorem~\ref{thm:intro-kcf-hypergraphs}, which asserts that for any hypergraph $H$ with
$n$ vertices and $m$ hyperedges, and for any $k > 1$, we have $\chi_{\mathrm{k-cf}}(H) = O(m^{\frac{1}{k+1}}\log^{\frac{k}{k+1}} n)$,
and that this bound is tight (even with respect to string graphs) up to logarithmic factors.

We shall use the following lemma that bounds the number of colors needed in a $(k+1)$-weak coloring of a hypergraph $H$ in terms of the maximum degree $\Delta(H)$.
\begin{lemma}\label{lem:non-mono-maxdegree}
Let $H=(V,\E)$ be a hypergraph, and let $\Delta$ denote the maximum degree of $H$. For any $k>1$, there exists a $(k+1)$-weak coloring of $V$ with $O(\Delta^{\frac{1}{k}})$ colors.
\end{lemma}

\begin{proof}
We first observe that w.l.o.g. we can assume that all hyperedges have cardinality $k+1$. Indeed, we can discard hyperedges of cardinality strictly less than $k+1$,
as this will not affect the desired coloring property. Moreover, we can replace any hyperedge with cardinality strictly more than $k+1$ with one of its subsets of cardinality $k+1$, since if such a subset is not monochromatic for a given coloring, then so is the superset. Note that in this process we might replace many hyperedges with one single hyperedge.
Obviously, the resulting hypergraph has maximum degree at most $\Delta$.

Next, we proceed with a probabilistic proof that will make use of the Lov\'{a}sz' Local Lemma (see, e.g., \cite{ALON00}).
We will use $M= A\Delta^{\frac{1}{k}}$ colors for some appropriate constant $A$ to be determined later. We color randomly, uniformly, and independently, each vertex $v \in V$ with a color so that each color is assigned to $v$ with probability $p = \frac{1}{M}$. Notice that a bad event $B_S$ that a given hyperedge $S \in \E$ is monochromatic happens with probability $p^k$. Note also that for every two hyperedges $S_1,S_2$ such that $S_1 \cap S_2 = \emptyset$, the corresponding events $B_{S_1}$ and $B_{S_2}$ are independent.
Thus for every hyperedge $S$, the event $B_S$ is independent of all but at most $(k+1)\Delta$ other events, as $S$ has a non-empty intersection with at most $(k+1)\Delta$ hyperedges.
We choose $A$ so that $ep^k[(k+1)\Delta +1] < 1$, and thus, $\frac{e[(k+1)\Delta +1]}{A^k\Delta} < 1$. (For example, $A =13$ is sufficient.)
Hence, by the symmetric version of the Lov\'{a}sz' Local Lemma, we have that:
$$
Pr[\bigcup_{S \in \E} B_S] < 1.
$$
In particular, this means that there exists a coloring for which none of the bad events happens. This completes the proof of the lemma.
\end{proof}

We are now ready to prove Theorem~\ref{thm:intro-kcf-hypergraphs}. Let us recall its statement.

\mn \textbf{Theorem~\ref{thm:intro-kcf-hypergraphs}}. Let $H=(V,\E)$ be an arbitrary hypergraph with $n$ vertices and $m$ hyperedges. For any $k > 1$, $$\chi_{\text{\textup{k-cf}}}(H) = O(m^{\frac{1}{k+1}}\log^{\frac{k}{k+1}} n).$$

\mn

\begin{proof}
Let $\Delta$ be an integer parameter to be determined later. Similarly to the proof of the upper bound on $\chi_{\mathrm{cf}}(H)$ for general hypergraphs by Cheilaris~\cite{CheilarisCUNYthesis2009} (presented in~\cite{CFPT09}), we start by iteratively coloring the vertices contained in more than $\Delta$ hyperedges.
Each time we find such a vertex, we color it with a unique color that we will not use again. We then remove this vertex together with all hyperedges containing it.
Note that in each step we remove at least $\Delta$ hyperedges so that the total number of colors we use is at most $\frac{m}{\Delta}$.

We are left with a hypergraph with at most $n$ vertices and at most $m$ hyperedges such that each vertex belongs to at most $\Delta$ hyperedges.
By Lemma~\ref{lem:non-mono-maxdegree}, this hypergraph admits a $(k+1)$-weak coloring with $O(\Delta^{\frac{1}{k}})$ colors. In the same fashion, any induced subhypergraph of $H$ admits such a $(k+1)$-weak coloring, and thus, by Lemma~\ref{lem:weakToKCF},
we can $k$-CF color this hypergraph using additional $O(\Delta^{\frac{1}{k}}\log n)$ colors.
The total number of colors we used is at most $O(\frac{m}{\Delta}+ \Delta^{\frac{1}{k}}\log n)$. Choosing $\Delta=\frac{m^{\frac{k}{k+1}}}{\log^{\frac{k}{k+1}} n}$, we obtain the asserted bound.
\end{proof}

As the neighborhood hypergraph of a graph on $n$ vertices has at most $n$ hyperedges, the following
corollary is immediate.

\begin{corollary}\label{cor:k-cf-graphs}
For any graph $G$ with $n$ vertices, $\kpn(G) =  O(n^{\frac{1}{k+1}}\log^{\frac{k}{k+1}} n)$.
\end{corollary}




The following proposition asserts that the upper bounds of Theorem~\ref{thm:intro-kcf-hypergraphs} and of Corollary~\ref{cor:k-cf-graphs} are tight up to a logarithmic factor.

\begin{proposition}\label{GBONCN}
Let $k\geq 1$ and let $G=(V,E)$ be any graph on $t\geq k+2$ vertices. Define the graph $G'$ by adding to $G$ a vertex $v_K$ for each $(k+1)$-tuple $K\subset V$ and connecting $v_K$ to all vertices of $K$. Let $H$ be the neighborhood hypergraph of $G'$. Then  $\chi_{\text{\textup{k-cf}}}(H) = \kpn(G')\geq \frac{t}{k}=\Omega( n^{\frac{1}{k+1}})$, where $n=\binom{t}{k+1}+t$ is the total number of vertices of $G'$ (and of $H$).
\end{proposition}

\begin{proof}
In any $k$-CF-coloring of $H'$,  a color can appear in $V$ at most $k$ times. (Otherwise, some $(k+1)$-tuple $K\subset V$  is monochromatic, and thus, the neighborhood of $v_K$ is also monochromatic, contradicting the definition of a $k$-CF-coloring.) Hence, $\chi_{\text{\textup{k-cf}}}(H) = \kpn(G')\geq t/k$. 
Note that $$n=t+\binom{t}{k+1}\leq 2\binom{t}{k+1} \leq 2\frac{t^{k+1}}{(k+1)!}\leq 2\frac{t^{k+1}}{e(\frac{k+1}{e})^{k+1}} \leq \left( \frac{e t}{k+1} \right)^{k+1},$$
where the first inequality holds, since $t \geq k+2$. Therefore,
$$\kpn(G) \geq \frac{t}{k} \geq \frac{1}{e} \frac{et}{k+1} \geq \frac{1}{e} n^{\frac{1}{k+1}} = \Omega \left(n^{\frac{1}{k+1}} \right),$$
as asserted.
\end{proof}

The lower bound presented in Proposition~\ref{GBONCN} also holds for string graphs, as it is easy to see that the construction described in the proof of Proposition~\ref{GBONCN} can be implemented by curves. In fact, the following claim shows that the upper bounds of Theorem~\ref{thm:intro-kcf-hypergraphs} and of Corollary~\ref{cor:k-cf-graphs} are tight up to logarithmic factors, even for a very specific class of string graphs.
\begin{definition}
An \emph{interval filament graph} is an intersection graph of curves defined as a strictly positive function on a closed interval, with zeros at the endpoints; that is, curves of the form $\{(t,f(t))\}_{a \leq t \leq b}$, where $f(a)=f(b)=0$ and $f(x)>0$ for all $a<x<b$. (Note that the endpoints of the interval may differ for different curves.)
\end{definition}

Let $\famI\famF^k(n)$ denote the largest $k$-CF chromatic number of a family $\famI\famF$ of $n$ interval filaments.
\begin{proposition}
\label{IntFil}
For any $k>1$, we have $\Omega (n^{\frac{1}{k+1}}) \leq \famI\famF^k(n) \leq O(n^{\frac{1}{k+1}}\log^{\frac{k}{k+1}} n)$. For $k=1$, $\famI\famF^1(n)=\Theta {(\sqrt n)}$.
\end{proposition}

\begin{proof}
Let $I_1 \subset \ldots  \subset I_t$  be closed intervals and $\{I'_{\{s_1,\ldots, s_{k+1}\}}:1 \leq  s_i \leq t, s_i \neq s_j$ for $ i\neq j\}$  be disjoint closed intervals contained in $I_1$ and indexed by  $(k+1)$-tuples representing $(k+1)$-subsets of $\{I_1, \ldots, I_t \}$. For each $(k+1)$-tuple $\{s_1,\ldots, s_{k+1}\}$ as above define the interval filament $F(I'_{\{s_1,\ldots, s_{k+1}\}})$ by a half circle centered in the middle of $I'_{\{s_1,\ldots, s_{k+1}\}}$ and contained in the upper halfplane. For each $1 \leq i \leq t$ define $F(I_i)$ to be a curve in the upper halfplane connecting the endpoints of $I_i$ and such that $F(I_i)$ intersects exactly those $F(I'_{\{s_1,\ldots, s_{k+1}\}})$ for which $i\in\{s_1,\ldots, s_{k+1}\}$ and is otherwise above the interval filaments $ F(I'_{\{s_1,\ldots, s_{k+1}\}})$. Clearly, we are in the setting of Proposition \ref{GBONCN}, which implies $\famI\famF^k(n)=\Omega(n^{\frac{1}{k+1}})$. The upper bound for $k>1$ follows from Corollary \ref{cor:k-cf-graphs}, and for $k=1$ from the general upper bound of~\cite{CFPT09} on the CF-chromatic numbers of hypergraphs.
\end{proof}

\subsection{k-CF-coloring of specific classes of string graphs}
\label{sec:string1}

Propositions~\ref{GBONCN} and~\ref{IntFil} show that, in general, the $k$-CF-chromatic number of a string graph may be as large as $\Omega(n^{\frac{1}{k+1}})$.
In this section, we show that for a large class of string graphs, $\kpn(G)$ is much lower, and can be bounded by $O(\log n)$, even for very small values of $k$. We prove a general result that allows such bounds to be obtained and then use it to prove Theorem~\ref{thm:framesLshapes}.

\begin{figure}[tb]
\label{fig:OBJ}
\begin{center}
\scalebox{0.6}{
\includegraphics{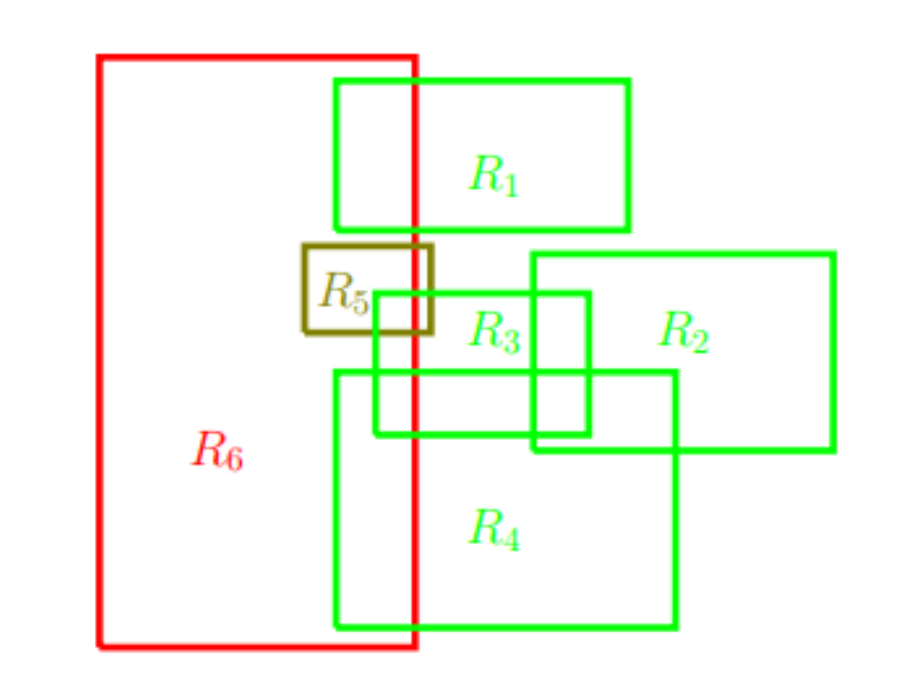}
} \caption{The hypergraph $H=(V,\famE)$, where $V=\{R_1,R_2,R_3,R_4,R_5\}$ and $\famE=\{\{R_3\},\{ R_1,R_3,R_4,R_5\}\}$ is the hypergraph arising from $\famF \times \famC$, where $\famF=\{R_1,R_2,R_3,R_4,R_5\}$ and $\famC=\{R_5,R_6\}.$}
\end{center}
\end{figure}

\medskip Our general result provides an upper bound on the $k$-CF chromatic number for families that satisfy a complex-looking technical condition, which turns out to be satisfied for numerous classes of string graphs. To present the result, we first introduce some definitions and notations, as well as assumptions that will be part of the technical condition.

For a collection $\famF$ of subsets of $\mathbb{R}^2$  and for a set $v\subset \mathbb{R}^2$, not necessarily in $\famF$, the notation $N_{\famF}(v)$ stands for the neighborhood of $v$ in the intersection graph of the family $\famF\cup \{v\}$.

For a collection $\famC$ of subsets of $\mathbb{R}^2$, the hypergraph $H$ arising from $\famF\times \famC$ has $\famF$ as its vertex set, and  each $\gamma\in\famC$ gives rise to the hyperedge $N_{\famF}(\gamma)$, as exemplified in Figure~6. Note that when $\famC=\famF$, the hypergraph $H$ is simply the neighborhood hypergraph of $\famF$.


We consider families $\famF,~ \famC$ of subsets of $\mathbb{R}^2$ such that each $\gamma \in \famF\cup \famC$ may be represented as a union of $\gamma_1$,$\ldots$, $\gamma_t$, where
\begin{enumerate}
\item Each $\gamma_i$ is either a curve or $\emptyset$,

\item Any two curves $\gamma_i,\gamma_j$ for $i\neq j$ are interior disjoint,

\item For any $\gamma,~\gamma' \in \famF\cup\famC$ such that $\gamma\neq \gamma'$, the curves $\gamma_i$, $\gamma'_i$ are disjoint for each $i$.
\end{enumerate}
In the sequel, we call these conditions \emph{the partition conditions}.

For example, if $\famF,~\famC$ are families of frames in the plane, then we can represent each $\gamma \in \famF$ as $\gamma=\gamma_1 \cup \ldots \cup \gamma_4$, where $\gamma_1$ is the left side, $\gamma_2$ is the bottom side, $\gamma_3$ is the right side, and $\gamma_4$ is the top side, and it is clear that the partition conditions are satisfied (assuming w.l.o.g. the appropriate general position that no two vertical or horizontal sides intersect.)


For families $\famF,~\famC$ as above, in order to obtain upper bounds on the $k$-CF chromatic number of the hypergraph $H$ arising from  $\famF\times \famC$,
 we introduce the concept of an \emph{intersection pattern} and its related definitions. A couple $(\gamma,\gamma')\in \famF\times \famC$, where $\gamma \neq \gamma'$,  is said to give an intersection pattern $(i,j)$ if $\gamma_i$ intersects $\gamma'_j$. A curve $\gamma'\in \famC$ is said to give $l$ intersection patterns with respect to $\famF$ if $|\{(i,j) :  \gamma_i\cap\gamma'_j \neq \emptyset \mbox{ for some } \gamma' \neq \gamma\in\famF\}|=l$.
 A product $\famF\times \famC$ is said to give at most $m$  intersection patterns if
 \[|\{(i,j) :  \gamma_i\cap\gamma'_j \neq \emptyset \mbox{ for some distinct } \gamma\in\famF,\gamma'\in\famC \}|\leq m.
 \]
 For example, if $\famF$ and $\famC$ are families of frames, and if we use the aforementioned partition of each frame $\gamma$ into $\gamma_1, \gamma_2, \gamma_3$, and $\gamma_4$, then
 the possible intersection patterns are $$\{2,3\},\{2,1\},\{4,3\},\{4,1\},\{3,4\},\{3,2\},\{1,4\},\{1,2\},$$ and any pair $(\gamma,\gamma')$ (where $\gamma \neq \gamma'$) with a non-empty intersection gives either 2 or 4 intersection patterns.

 If $|\famF|=n$ and $\famF\times \famC$ gives at most $m$ intersection patterns, then we call $\famF\times \famC$ an \emph{$m$-family of size $n$}.

 Finally, two curves $\gamma,\delta$ are said to be \emph{consecutive} along a curve $\eta$ with respect to some family of curves $\famG$  if there exist $x \in \gamma \cap \eta$  and $y \in \delta \cap \eta$ such that the interior of the subcurve of  $\eta$ connecting $x$ and $y$ does not intersect any curve of $\famG$.

Now, we are ready to state our general result.


\begin{theorem}
\label{tlcs}
Let $\famF\times \famC$ be an $m$-family of size $n$ such that:
\begin{itemize}
\item Each $\gamma'\in \famC$ gives at most $ls$ intersection patterns, and

\item Any couple $(\gamma,\gamma')\in \famF\times \famC $ with $\gamma\cap\gamma'\neq \emptyset$  gives at least $l$ intersection patterns.
\end{itemize}
Then, the hypergraph $H$  arising from $\famF\times \famC$ can be $s$-CF-colored with $O(m\log n)$ colors.
\end{theorem}

\begin{proof}
For the sake of convenience, we break the proof of the theorem into several claims.

Let $\mathcal{K}\subset \mathcal{F}$. For each intersection pattern $(i,j)$ of $\famF\times \famC$, we define $G_{(i,j)}(\mathcal{K})$ to be the graph whose vertex set is $\{\gamma_i:\gamma \in \mathcal{K}\}$, where there is an edge between two curves $\gamma_i$ and $\tilde{\gamma}_i$ if and only if they are consecutive along the subcurve $\gamma'_j$ of some curve $\gamma'\in\mathcal{C}$, with respect to $\{\delta_i: \delta \in \mathcal{K}\}$.

Note that if $\gamma'=\gamma$ or $\gamma' = \tilde{\gamma}$, then $\gamma_i,\tilde{\gamma}_i$ are not consecutive along $\gamma'_j$, by the definition of the notion `consecutive'. Hence, even if $\famF \cap \famC \neq \emptyset$, we do not have to consider pairs $(\gamma,\gamma')$ where $\gamma'=\gamma$. This justifies the exclusion of such pairs from the definition of intersection patterns above.
\begin{claim}\label{Claim:G_i,j-planar}
Each graph $G_{(i,j)}(\mathcal{K})$ is planar.
\end{claim}

\begin{proof}[Proof of Claim~\ref{Claim:G_i,j-planar}]
We derive the claim from Lemma~\ref{lem:prelim-strings-planar}. To this end, we first set $S_1$ (in the notations of Lemma~\ref{lem:prelim-strings-planar}) to be the collection of curves $\gamma_i$. Then, for each edge $\{\gamma_i,\tilde{\gamma}_i\}$  in $G_{(i,j)}(\mathcal{K})$, we consider the appropriate subcurve of the curve $\gamma'_j$ along which $\gamma_i$ and $\tilde{\gamma}_i$ are consecutive; we then perform on these subcurves a small perturbation and define $S_2$ to be the family of obtained curves. As the curves $\{\gamma'_j\}$ are pairwise disjoint, it is clear that the perturbation can be performed in such a way that $(S_1,S_2)$ satisfy the conditions of Lemma~\ref{lem:prelim-strings-planar}. Therefore, by Lemma~\ref{lem:prelim-strings-planar}, $G_{(i,j)}(\mathcal{K})$ is planar, as asserted.
\end{proof}

Now, let $G(\mathcal{K})$ be the graph whose vertex set is $\mathcal{K}$, where there is an edge between $\gamma,\tilde{\gamma} \in \mathcal{K}$ if some $G_{(i,j)}(\mathcal{K})$ contains the edge $\{\gamma_i, \tilde{\gamma}_i\}$.
\begin{claim}\label{Claim:Proper-Aux}
$G(\mathcal{K})$ is $6m$-colorable.
\end{claim}
\begin{proof}[Proof of Claim~\ref{Claim:Proper-Aux}]
As each graph $G_{(i,j)}(\mathcal{K})$ is planar, and there are at most $m$ such graphs, $G(\mathcal{K})$ is a union of at most $m$ planar graphs.
By Claim~\ref{Cl:union-of-planar}, $G_{(i,j)}(\mathcal{K})$ is $6m$-colorable, as asserted.
\end{proof}


Recall that an $r$-weak coloring of a hypergraph is a coloring of its vertices such that no hyperedge of size $\geq r$ is monochromatic.


\begin{claim}
\label{NMC}
In the notations of Theorem~\ref{tlcs}, let $H'$ be any induced subhypergraph of $H$ arising from $\mathcal{K}\times\famC$, where  $\mathcal{K}\subset \famF$. Then, for any curve  $\gamma'\in\famC$ such that $\gamma'$ gives at most $ls'$ intersection patterns and $ |N_{\mathcal{K}}(\gamma')|\geq s'+1$, the hyperedge $ N_{\mathcal{K}}(\gamma')$  of $H'$ contains an edge of $G(\mathcal{K})$ as a subset. In particular, when $s'=s$,  any proper coloring $c$ of $G(\mathcal{K})$ is an $(s+1)$-weak coloring of $H'$.
\end{claim}

\begin{proof}[Proof of Claim~\ref{NMC}.]
We present the proof for $H$; it will be apparent that the same argument also holds for any induced subhypergraph.

Let $e$ be a hyperedge of $H$ of size $\geq s'+1$ such that $e=N_{\famF}(\gamma')$ for some $\gamma' \in \famC$ that gives at most $ls'$ intersection patterns. Recall that each of the $|N_{\mathcal{F}}(\gamma')| \geq s'+1$ intersecting pairs gives at least $l$ intersection patterns.
By the pigeonhole principle, some intersection pattern $(i,j)$ arises from two different couples $(\gamma,\gamma') $ and $(\delta,\gamma')$, and so both $\gamma_i$ and $\delta_i$ intersect $\gamma'_j$. We can assume w.l.o.g. that $\gamma_i$ and $\delta_i$ are consecutive along $\gamma'_j$ with respect to $\{\eta_i:\eta\in \famF\}$. (Indeed, if there are more than two subcurves in $\famF$ that intersect $\gamma'_j$, we can take two of them that are consecutive along $\gamma'_j$.)
Thus, we have $(\gamma_i,\delta_i) \in  G_{(i,j)}(\famF)$, and hence, $(\gamma , \delta) \in G(\famF)$.
 Finally, as Claim  \ref{Claim:Proper-Aux} holds for any $\mathcal{K} \subset \famF$, the same argument applies for any induced subhypergraph of $H$. This completes the proof of Claim~\ref{NMC}.
\end{proof}

Claim \ref{NMC} implies that when $s'=s$, any proper coloring of $G(\mathcal{K})$ is an $(s+1)$-weak coloring of the hypergraph arising from $\mathcal{K}\times\famC$,  where $\mathcal{K}\subset \famF$. Therefore, by Claim \ref{Claim:Proper-Aux}, any induced subhypergraph of $H$ admits an $(s+1)$-weak coloring with at most $6m$ colors.
  By Lemma~\ref{lem:weakToKCF}, it follows that $\chi_{\text{\textup{s-cf}}}(H) = O(m \log n)$. This completes the proof of Theorem~\ref{tlcs}.
\end{proof}

 We are now ready to prove Theorem~\ref{thm:framesLshapes}. Let us recall its statement.

\mn \textbf{Theorem~\ref{thm:framesLshapes}.}
For a class $\mathcal{F}$ and for $k \in \mathbb{N}$, denote by $\mathcal{F}^k(n)$ the largest $k$-CF chromatic number of an intersection graph of a family of $n$ elements of $\mathcal{F}$. Then:
\begin{enumerate}
\item For the class $\famL$ of L-shapes, $\famL^k(n)= \Theta(\log n)$ for all fixed $k \geq 2$.

\item For the class $\famF$ of frames, $\famF^k(n)= \Theta(\log n)$ for all fixed $k \geq 4$.
\end{enumerate}

\begin{proof}
Each L-shape $\gamma$ can be naturally represented as $\gamma = \gamma_1 \cup \gamma_2$, where $\gamma_1$ is its vertical part and $\gamma_2$ is its horizontal part. Clearly, any family of L-shapes satisfies the aforementioned \emph{partition conditions} with respect to this partition.

For any family $\famL $ of $n$ L-shapes, the pair $\famL\times \famL$ gives at most $2$ intersection patterns, while any $(\gamma,\gamma')\in \famL\times \famL$ such that $\gamma\cap\gamma'\neq \emptyset$ obviously gives at least one intersection pattern. Therefore, by Theorem~\ref{tlcs}, the hypergraph arising from $\famL\times \famL$,  which is simply the neighborhood hypergraph of $\famL$, can be $2$-CF colored with $O(\log n)$ colors. Since a $2$-CF coloring is also a $k$-CF coloring for any $k\geq 2$, the claimed upper bound on $\famL^k(n)$ follows.

Analogously, using the aforementioned representation of each frame $\gamma$ as a union $\gamma_1 \cup \ldots \cup \gamma_4$, for any family $\famF $ of $n$ frames, the pair $\mathcal{F}\times \mathcal{F}$ gives at most $8$ intersection patterns, while any $(\gamma,\gamma')\in \mathcal{F}\times \mathcal{F}$ such that $\gamma\cap\gamma'\neq \emptyset$ gives at least $2$ intersection patterns. Therefore, by Theorem~\ref{tlcs}, the hypergraph arising from $\mathcal{F}\times \mathcal{F}$, which is simply the neighborhood hypergraph of $\famF$, can be $4$-CF colored  with $O(\log n)$ colors. As before, the claimed upper bound for all $k \geq 4$ follows.

Finally, the lower bounds follow from a combination of Claim~\ref{cl:olap_is_frames} and Lemma~\ref{LBint}.
\end{proof}

\begin{remark}
\label{rem:gLs}
As was mentioned at the beginning of Section~\ref{sec:kcf}, $\famF^1(n)= \Theta( \sqrt n)$ and $\famL^1(n)= \Theta( \sqrt n)$. Hence, the problem of $k$-CF coloring of intersection graphs of frames and of L-shapes is now asymptotically solved for any fixed $k$. Furthermore, by Claim~\ref{cl:olap_is_frames} and Lemma~\ref{LBint}, the tight result $\famL^k(n)= \Theta(\log n)$ for all $k \geq 2$, also holds if we replace ``L-shapes'' by ``grounded L-shapes''.
\end{remark}

\subsection{Bounding the CF-chromatic number of string graphs in terms of their chromatic number}
\label{sec:string_bounded_chi}


For general graphs, having a small chromatic number does not imply any bound on the CF-chromatic number. For example, if $G$ is the 1-subdivision of $K_m$ (i.e., a graph $G$ obtained from $K_m$ by adding a vertex in the `middle' of each edge, thus dividing each edge into two edges), then $G$ is bipartite and thus 2-colorable, while $\pn(G) = \Theta(\sqrt{|V(G)|})$.
The above Theorem~\ref{tlcs} allows us to prove Theorem~\ref{thm:intro-string-graphs-UB}, which implies that for string graphs, the situation is significantly different. We prove the following extended version of the theorem.

\begin{proposition}[Extended version of Theorem~\ref{thm:intro-string-graphs-UB}]\label{Prop:Ext-thm-bounding-chromatic}
Let $G=(V,E)$ be a string graph on a set $S$ of $n$ strings such that $\chi(G) \leq t$. Then $\pn(G) = O(t^2 \log n)$. In particular, any string graph $G$ with clique number $\omega$ admits a  CF-coloring with $O(\log^{c \omega} n)$ colors, where $c$ is some constant.
On the other hand, there exists a 2-colorable string graph $G$ such that $\pn(G) = \Omega(\log n)$.
\end{proposition}




\medskip \noindent In the proof of the theorem, we use a general coloring result for string graphs with clique number at most $\omega$ obtained by Fox and Pach~\cite{FP14}.
\begin{theorem}[Fox and Pach]
\label{FP}
Let $\famF$ be a family of $n$  curves with clique number at most $\omega$. Then $\famF$ can be properly colored with $\log^{C\omega} n$ colors, for some constant $C>0$.
\end{theorem}

\begin{proof}[Proof of Theorem~\ref{thm:intro-string-graphs-UB}]
By the assumption on $\chi(G)$, there exists a partition of $\famS$ into $t$ subsets $\famS_1, \ldots, \famS_t$, such that the curves inside each $\famS_i$ are pairwise disjoint. To apply Theorem~\ref{tlcs}, we represent each $s \in \famS$ as $s=s_1 \cup s_2 \cup \ldots \cup s_t$, where for each $s \in \famS_i$, we set $s_i=s$ and $s_j=\emptyset$ for all $j\neq i$. It is clear that with respect to this representation, the family $\famS$ satisfies the aforementioned \emph{partition conditions}. Furthermore, for each hypergraph $H_i$ arising from $\famS_i\times\bigcup_{j\neq i} \famS_j$, any $s\in \bigcup_{j\neq i} \famS_j$ gives at most one intersection pattern, while $\famS_i\times\bigcup_{j\neq i} \famS_j$ gives at most $t$-$1$ intersection patterns and is thus  a $(t-1)$-family. Therefore, by Theorem~\ref{tlcs}, each $H_i$ can be CF-colored with $O(t\log n)$ colors. Using a different palette of $O(t\log n)$ colors for each $H_i$, we can color the entire $\famS$ using $O(t^2\log n)$ colors.

To see that the coloring we obtain is a CF-coloring of $\famS$, take any $s \in \famS$ with a non-empty neighborhood and any of its neighbors $s'$. For some $i \neq j $, we have $s \in \famS_j$ and  $s' \in \famS_i$. Hence, the neighborhood of $s$ restricted to $\famS_i$ is a non-empty hyperedge of $H_i$, and thus, some color appears in that neighborhood exactly once. Since that color does not appear in the coloring of $H_l$ for $l \neq i$, our coloring is indeed a CF-coloring of $\famS$.

The second part of the theorem follows immediately by applying Theorem \ref{FP}. We simply substitute $\log^{C\omega} n$ in place of $t$, where $C$ is the constant of Theorem \ref{FP}, yielding the upper bound $\log^{2C\omega} n\log n=\log^{2C\omega+1} n\leq\log^{\tilde{C}\omega} n$ for  $\tilde{C}=2C+1$ on the CF-chromatic number of string graphs with clique number at most $\omega$.

The tightness direction follows from the proof of the lower bound of Theorem~\ref{CFint} presented in Section~\ref{sec:circle_graphs}. (In that proof, we present a circle graph whose $k$-CF chromatic number is $\Omega(\log n)$ for any fixed $k$; it is clear from the construction that the graph is bipartite, and so 2-colorable.)
\end{proof}

Finally, for many specific families of string graphs the bound $\log^{\tilde{C}\omega} n$ can be further improved. For example, for families of  $x$-monotone curves (i.e., curves that are intersected by any vertical line
at most once), the following corollary holds:
\begin{corollary}
Let $\famF$ be a family of $n$  $x$-monotone curves with a bounded clique number. Then $\pn(\famF)= O(\log^3 n)$.
\end{corollary}
\begin{proof}
It was proved in~\cite{RW14} that we have $\chi(\famF)=O(\log n)$. Hence, by Theorem~\ref{thm:intro-string-graphs-UB}, $\pn(\famF)= O(\log^3 n)$.
\end{proof}

\section{Discussion}
\label{sec:discussion}
CF-coloring problems arising from string graphs are similar, in some sense, to the well-studied problem of $\chi$-boundedness in the context of string graphs.

The class of string graphs is not $\chi$-bounded, due to the high freedom one has to place the curves in the plane. In fact, even the class of intersection graphs of segments is not $\chi$-bounded (see~\cite{PKKLMTW14}). Similarly, as we see in this paper, many specific classes of string graphs do not admit any better upper bound on their CF and $k$-CF chromatic numbers than the one that already holds in the non-geometric context.

For more restricted classes of string graphs, various results regarding $\chi$-boundness have been obtained along the years. For example, a double exponential upper bound on the largest chromatic number of any interval overlap graph with a fixed clique number was proved by Gy\'arf\'as in~\cite{Gy85} (see also~\cite{Ko88}). In 1996, McGuinness~\cite{MG96} showed that the intersection graph of grounded L-shapes is also $\chi$-bounded.
Similarly, in the context of CF-coloring, Theorem~\ref{CFint} implies that any interval overlap graph (=circle graph) has at most a logarithmic CF-chromatic number, and Theorem~\ref{thm:L-shapes} bounds the CF-chromatic number of the intersection graph of grounded L-shapes.

Gy\'arf\'as' aforementioned result establishes an interesting framework further developed in  \cite{LW14,MG96,RW14,SU14}. The culmination of this line of research is~\cite{RW17}, where it is shown that the class of intersection graphs arising from curves intersecting a fixed curve at least once and at most a constant number of times is $\chi$-bounded.

In light of the results on $\chi$-boundedness and of our results on interval overlap graphs and on grounded L-shapes, it is natural to ask whether families of curves intersecting a fixed line a bounded number of times are also special in the CF or the $k$-CF-coloring sense. More precisely, it may be asked whether it is true that such families of $n$ curves can be CF-colored with polylog($n$) colors, or at least $k$-CF colored with polylog($n$) colors for a sufficiently large constant $k$, like in the cases of L-shapes and frames.

Unfortunately, this is not the case. While the $\chi$-boundness of interval filament graphs was proved in~\cite{KW14}, Proposition~\ref{IntFil} demonstrates that such graphs do not admit any upper bound on their  $k$-CF chromatic number that is significantly better than the general upper bound.

While the ``$k$-CF" variant of the result of \cite{RW17} does not hold, other possible generalizations may be considered. In particular, note that in any representation of interval filament graphs leading to Proposition \ref{IntFil}, the maximum number of intersections between two curves in the family is clearly $\Omega(n^{\frac{k}{k+1}})$. It seems that producing string graphs with a large CF-chromatic number requires extensive of freedom to place the curves or many intersections between some pairs of curves. It may be the case that bounding both the maximum number of intersections between a pair of curves and the maximum number of intersections of each curve with a fixed line (but imposing, of course, at least one intersection) is the right condition for obtaining some ``interesting'' upper bound on the CF-chromatic number of the corresponding intersection graph.

\bibliographystyle{plain}
\bibliography{references}

\begin{thebibliography}{10}

\bibitem{AADFGHKS17}
Z.~Abel, V.~Alvarez, E.D. Demaine, S.~P. Fekete, A.~Gour, A.~Hesterberg,
  P.~Keldenich, and C.~Scheffer.
\newblock Three colors suffice: Conflict-free coloring of planar graphs.
\newblock In {\em Proc. 28th Annu. ACM-SIAM Sympos. Discrete Algorithms}, pages
  1951--1963, 2017.

\bibitem{AS08}
N.~Alon and S.~Smorodinsky.
\newblock Conflict-free colorings of shallow discs.
\newblock {\em Internat. J. Comput. Geom. Appl.}, 18(6):599--604, 2008.

\bibitem{ALON00}
N.~Alon and J.~H. Spencer.
\newblock {\em The Probabilistic Method}.
\newblock Wiley Inter-Science, 2nd edition, 2000.

\bibitem{BT16}
Kevin Balas and Csaba~D. T{\'{o}}th.
\newblock On the number of anchored rectangle packings for a planar point set.
\newblock {\em Theoretical Computer Science}, 654:143 -- 154, 2016.

\bibitem{Ben59}
S.~Benzer.
\newblock On the topology of the genetic fine structure.
\newblock {\em Proc. Natl. Acad. Sci.}, 45:1607--1620, 1959.

\bibitem{Chan12}
T.M. Chan.
\newblock Conflict-free coloring of points with respect to rectangles and
  approximation algorithms for discrete independent set.
\newblock In {\em Symposuim on Computational Geometry 2012, SoCG '12, Chapel
  Hill, NC, USA, June 17-20, 2012}, pages 293--302, 2012.

\bibitem{CU13}
Steven Chaplick and Torsten Ueckerdt.
\newblock Planar graphs as vpg-graphs.
\newblock In Walter Didimo and Maurizio Patrignani, editors, {\em Graph
  Drawing}, pages 174--186, Berlin, Heidelberg, 2013. Springer Berlin
  Heidelberg.

\bibitem{CheilarisCUNYthesis2009}
P.~Cheilaris.
\newblock {\em Conflict-free coloring}.
\newblock PhD thesis, City University of New York, 2009.

\bibitem{cf7}
K.~Chen, A.~Fiat, M.~Levy, J.~Matou{\v s}ek, E.~Mossel, J.~Pach, M.~Sharir,
  S.~Smorodinsky, U.~Wagner, and E.~Welzl.
\newblock Online conflict-free coloring for intervals.
\newblock {\em SIAM J. Comput.}, 36:545--554, 2006.

\bibitem{CKS2009talg}
K.~Chen, H.~Kaplan, and M.~Sharir.
\newblock Online conflict free coloring for halfplanes, congruent disks, and
  axis-parallel rectangles.
\newblock {\em ACM Transactions on Algorithms}, 5(2):16.1--16.24, 2009.

\bibitem{Cho34}
Ch. Chojanski.
\newblock {\"{U}}ber wesentlich unpl{\"{a}}ttbare kurven im dreidimensionalen
  raume.
\newblock {\em Fundamenta Mathematicae}, 23(1):135–142, 1934.

\bibitem{CSS16}
M.~Chudnovsky, A.~Scott, and P.~Seymour.
\newblock Induced subgraphs of graphs with large chromatic number. {V}.
  chandeliers and strings.
\newblock Manuscript, 2016.

\bibitem{EET76}
G.~Ehrlich, S.~Even, and R.~E. Tarjan.
\newblock Intersection graphs of curves in the plane.
\newblock {\em J. Combin. Theory Ser. B}, 21:8--20, 1976.

\bibitem{ELRS}
G.~Even, Z.~Lotker, D.~Ron, and S.~Smorodinsky.
\newblock Conflict-free colorings of simple geometric regions with applications
  to frequency assignment in cellular networks.
\newblock {\em SIAM J. Comput.}, 33:94--136, 2003.

\bibitem{FK17}
S.P. Fekete and P.~Keldenich.
\newblock Conflict-free coloring of intersection graphs.
\newblock In Yoshio Okamoto and Takeshi Tokuyama, editors, {\em 28th
  International Symposium on Algorithms and Computation, {ISAAC} 2017, December
  9-12, 2017, Phuket, Thailand}, volume~92 of {\em LIPIcs}, pages 31:1--31:12.
  Schloss Dagstuhl - Leibniz-Zentrum fuer Informatik, 2017.

\bibitem{FKMU16}
Stefan Felsner, Kolja Knauer, George~B. Mertzios, and Torsten Ueckerdt.
\newblock Intersection graphs of {L}-shapes and segments in the plane.
\newblock {\em Discrete Applied Mathematics}, 206:48 -- 55, 2016.

\bibitem{FP10}
J.~Fox and J.~Pach.
\newblock A separator theorem for string graphs and its applications.
\newblock {\em Combinatorics, Probability {\&} Computing}, 19(3):371--390,
  2010.

\bibitem{FP12}
J.~Fox and J.~Pach.
\newblock String graphs and incomparability graphs.
\newblock {\em Advances in Mathematics}, 230:1381--1401, 2012.

\bibitem{FP14}
J.~Fox and J.~Pach.
\newblock Applications of a new separator theorem for string graphs.
\newblock {\em Combinatorics, Probability {\&} Computing}, 23(1):66--74, 2014.

\bibitem{GR15}
L.~Gargano and A.~A. Rescigno.
\newblock Complexity of conflict-free colorings of graphs.
\newblock {\em Theor. Comput. Sci.}, 566:39--49, 2015.

\bibitem{gibsonvar}
M.~Gibson and K.~R. Varadarajan.
\newblock Decomposing coverings and the planar sensor cover problem.
\newblock In {\em 50th Annual IEEE Symposium on Foundations of Computer
  Science, FOCS 2009, October 25-27, 2009, Atlanta, Georgia, USA}, pages
  159--168, 2009.

\bibitem{GST14}
R.~Glebov, T.~Szab{\'{o}}, and G.~Tardos.
\newblock Conflict-free colouring of graphs.
\newblock {\em Combinatorics, Probability {\&} Computing}, 23(3):434--448,
  2014.

\bibitem{Gy85}
A.~Gy{\'{a}}rf{\'{a}}s.
\newblock On the chromatic number of multiple interval graphs and overlap
  graphs.
\newblock {\em Discrete Mathematics}, 55(2):161--166, 1985.

\bibitem{HS02}
S.~Har-Peled and S.~Smorodinsky.
\newblock Conflict-free coloring of points and simple regions in the plane.
\newblock {\em Discrete Comput. Geom.}, 34(1):47--70, 2005.

\bibitem{hks09}
E.~Horev, R.~Krakovski, and S.~Smorodinsky.
\newblock Conflict-free coloring made stronger.
\newblock In {\em Algorithm Theory - {SWAT} 2010, 12th Scandinavian Symposium
  and Workshops on Algorithm Theory, Bergen, Norway, June 21-23, 2010.
  Proceedings}, pages 105--117, 2010.

\bibitem{KS18}
C.~Keller and S.~Smorodinsky.
\newblock Conflict-free coloring of intersection graphs of geometric objects.
\newblock In {\em Proceedings of the Twenty-Ninth Annual {ACM-SIAM} Symposium
  on Discrete Algorithms, {SODA} 2018, New Orleans, LA, USA, January 7-10,
  2018}, pages 2397--2411, 2018.

\bibitem{Kes2}
B.~Keszegh.
\newblock Weak conflict-free colorings of point sets and simple regions.
\newblock In {\em Proceedings of the 19th Annual Canadian Conference on
  Computational Geometry, CCCG 2007, August 20-22, 2007, Carleton University,
  Ottawa, Canada}, pages 97--100, 2007.

\bibitem{Kes1}
B.~Keszegh.
\newblock {\em Combinatorial and computational problems about points in the
  plane}.
\newblock PhD thesis, Central European University, Budapest, Department of
  Mathematics and its Applications, 2009.

\bibitem{Kes17+}
B.~Keszegh.
\newblock Coloring intersection hypergraphs of pseudo-disks.
\newblock In Bettina Speckmann and Csaba~D. T{\'{o}}th, editors, {\em 34th
  International Symposium on Computational Geometry, SoCG 2018, June 11-14,
  2018, Budapest, Hungary}, volume~99 of {\em LIPIcs}, pages 52:1--52:15.
  Schloss Dagstuhl - Leibniz-Zentrum fuer Informatik, 2018.

\bibitem{Ko88}
A.~V. Kostochka.
\newblock Upper bounds on the chromatic number of graphs.
\newblock {\em Trudy Inst. Mat. (Novosibirsk)}, 10(Modeli i Metody
  Optim.):204--226, 265, 1988.

\bibitem{Ko97}
A.~V. Kostochka and Jan Kratochv{\'{\i}}l.
\newblock Covering and coloring polygon-circle graphs.
\newblock {\em Discrete Mathematics}, 163(1-3):299--305, 1997.

\bibitem{Ko98}
A.~V. Kostochka and J.~Nesetril.
\newblock Coloring relatives of intervals on the plane, {I:} chromatic number
  versus girth.
\newblock {\em Eur. J. Comb.}, 19(1):103--110, 1998.

\bibitem{Ko01}
A.~V. Kostochka and J.~Nesetril.
\newblock Colouring relatives of intervals on the plane, {II:} intervals and
  rays in two directions.
\newblock {\em Eur. J. Comb.}, 23(1):37--41, 1998.

\bibitem{KM91}
J.~Kratochv{\'{i}}l and J.~Matou{\v{s}}ek.
\newblock String graphs requiring exponential representations.
\newblock {\em J. Combin. Theory Ser. B}, 53:1--4, 1991.

\bibitem{Wa15}
T.~Krawczyk, A.~Pawlik, and B.~Walczak.
\newblock Coloring triangle-free rectangle overlap graphs with o(log log n)
  colors.
\newblock {\em Discrete {\&} Computational Geometry}, 53(1):199--220, 2015.

\bibitem{KW14}
T.~Krawczyk and B.~Walczak.
\newblock Coloring relatives of interval overlap graphs via on-line games.
\newblock {\em Automata, Languages, and Programming - 41st International
  Colloquium, {ICALP} 2014, Copenhagen, Denmark, July 8-11, 2014, Proceedings,
  Part {I}}, pages 738--750, 2014.

\bibitem{KW16}
T.~Krawczyk and B.~Walczak.
\newblock On-line approach to off-line coloring problems on graphs with
  geometric representations.
\newblock {\em Combinatorica}, Dec 2016.

\bibitem{LW14}
M.~Lason, P.~Micek, A.~Pawlik, and B.~Walczak.
\newblock Coloring intersection graphs of arc-connected sets in the plane.
\newblock {\em Discrete Comput. Geom.}, 2:399--415, 2014.

\bibitem{Matousek14}
J.~Matou{\v{s}}ek.
\newblock Near-optimal separators in string graphs.
\newblock {\em Combinatorics, Probability {\&} Computing}, 23(1):135--139,
  2014.

\bibitem{MCG96}
S.~McGuinness.
\newblock On bounding the chromatic number of {L}-graphs.
\newblock {\em Discrete Mathematics}, 154(1):179 -- 187, 1996.

\bibitem{MG96}
S.~McGuinness.
\newblock On bounding the chromatic number of {L}-graphs.
\newblock {\em Discrete Mathematics}, 154(1-3):179--187, 1996.

\bibitem{MG00}
S.~McGuinness.
\newblock Colouring arcwise connected sets in the plane {I}.
\newblock {\em Graphs and Combinatorics}, 16(4):429--439, 2000.

\bibitem{CFPT09}
J.~Pach and G.~Tardos.
\newblock Conflict-free colourings of graphs and hypergraphs.
\newblock {\em Combin. Probab. Comput.}, 18(5):819--834, 2009.

\bibitem{pach}
J.~Pach, G.~Tardos, and G.~T\'{o}th.
\newblock Indecomposable coverings.
\newblock In {\em The China-Japan Joint Conference on Discrete Geometry,
  Combinatorics, and Graph Theory (CJCDGCGT 2005), Lecture Notes in Computer
  Sceince}, pages 135--148, 2007.

\bibitem{PT02}
J.~Pach and G.~T{\'{o}}th.
\newblock Recognizing string graphs is decidable.
\newblock {\em Discrete Comput. Geom.}, 28:593--606, 2002.

\bibitem{PachToth}
J.~Pach and G.~T{\'o}th.
\newblock Decomposition of multiple coverings into many parts.
\newblock {\em Comput. Geom.: Theory and Applications}, 42(2):127--133, 2009.

\bibitem{PKKLMTW14}
A.~Pawlik, J.~Kozik, T.~Krawczyk, M.~Lason, P.~Micek, W.T. Trotter, and
  B.~Walczak.
\newblock Triangle-free intersection graphs of line segments with large
  chromatic number.
\newblock {\em J. Comb. Theory, Ser. {B}}, 105:6--10, 2014.

\bibitem{RW14}
A.~Rok and B.~Walczak.
\newblock Outerstring graphs are {\(\chi\)}-bounded.
\newblock In {\em 30th Annual Symposium on Computational Geometry, SOCG'14,
  Kyoto, Japan, June 08 - 11, 2014}, page 136, 2014.

\bibitem{RW17}
A.~Rok and B.~Walczak.
\newblock Coloring curves that cross a fixed curve.
\newblock In {\em 33rd International Symposium on Computational Geometry, SoCG
  2017, July 4-7, 2017, Brisbane, Australia}, pages 56:1--56:15, 2017.

\bibitem{SSS03}
M.~Schaefer, E.~Sedgwick, and D.~{\v{S}}tefankovi{\v{c}}.
\newblock Recognizing string graphs in {NP}.
\newblock {\em J. Comput. System Sci., special issue of STOC'2002},
  67:593--606, 2003.

\bibitem{Sherwani99}
N.~A. Sherwani.
\newblock {\em Algorithms for VLSI Physical Design Automation}.
\newblock Springer US, third edition, 1999.

\bibitem{SmPHD}
S.~Smorodinsky.
\newblock {\em Combinatorial Problems in Computational Geometry}.
\newblock PhD thesis, School of Computer Science, Tel-Aviv University, 2003.

\bibitem{smoro}
S.~Smorodinsky.
\newblock On the chromatic number of some geometric hypergraphs.
\newblock {\em SIAM J. Discrete Math.}, 21:676--687, 2007.

\bibitem{CF-survey}
S.~Smorodinsky.
\newblock {\em Conflict-Free Coloring and its Applications, Geometry ---
  Intuitive, Discrete, and Convex}, pages 331--389.
\newblock Springer Berlin Heidelberg, Berlin, Heidelberg, 2013.

\bibitem{SMO10}
Shakhar Smorodinsky.
\newblock Conflict-free coloring and its applications.
\newblock {\em CoRR}, abs/1005.3616, 2010.

\bibitem{SU14}
A.~Suk.
\newblock Coloring intersection graphs of x-monotone curves in the plane.
\newblock {\em Combinatorica}, 34(4):487--505, 2014.

\bibitem{Tut70}
W.~T. Tutte.
\newblock Toward a theory of crossing numbers.
\newblock {\em Journal of Combinatorial Theory}, 8:45--53, 1970.

\end{thebibliography}

\end{document}